\documentclass[10pt]{amsart}

\usepackage{amsmath,amssymb,amsthm}
\usepackage{array,multirow}
\usepackage{multicol}
\usepackage{enumerate} 
\usepackage{cases} 
\usepackage[unicode,linkcolor=blue,hyperindex]{hyperref}
\usepackage{graphicx}
\usepackage{mathtools}
\usepackage{cleveref}
\usepackage{todonotes}

\numberwithin{equation}{section} \numberwithin{figure}{section}

{\theoremstyle{remark} \newtheorem{remark}{Remark}}
\newtheorem{definition}{Definition}[section]
\newtheorem{proposition}{Proposition}[section]

\newtheorem{theorem}{Theorem}[section]
\newtheorem{corollary}{Corollary}[section]
\newtheorem{lemma}{Lemma}[section]
{\theoremstyle{remark} \newtheorem{example}{Example}[section]}
\newtheorem{hypothesis}{Hypothesis}[section]

\newcommand\Lap{(-\Delta)^s_p}

\newcommand{\SZ}{{\Pi_h}}

\def\calA{\mathcal{A}}
\def\calF{\mathcal{F}}

\def\calO{\mathcal{O}}
\def\calC{\mathcal{C}}
\def\calB{\mathcal{B}}

\newcommand{\Th}{{\mathcal {T}_h}}

\newcommand{\N}{\mathbb{N}}
\newcommand{\R}{\mathbb{R}}
\newcommand{\Rd}{{\mathbb{R}^d}}
\newcommand{\vn}{\textbf{n}}
\newcommand{\phii}{\varphi}
\newcommand{\pp}{\partial}
\newcommand{\eps}{\varepsilon}
\newcommand{\x}{\texttt{x}}
\newcommand{\wt}{\widetilde}
\newcommand{\dist}{\textrm{dist}}

\newcommand{\vp}{\varphi}

\newcommand{\Wzop}{ \mbox{ \raisebox{7.2pt} {\tiny$\circ$} \kern-10.7pt} {W}^1_p }
\newcommand{\supp}{\textrm{supp~}}
\newcommand{\iii}[1]{{\left\vert\kern-0.25ex\left\vert\kern-0.25ex\left\vert #1 
    \right\vert\kern-0.25ex\right\vert\kern-0.25ex\right\vert}}
\newcommand{\scalardelta}[1]{{( \kern-0.25ex ( #1 )\kern-0.25ex )}}

\begin{document} 
\title[Quasi-linear fractional-order operators]{Quasi-linear fractional-order operators in Lipschitz domains}

\begin{abstract}
  We prove Besov boundary regularity for solutions of the homogeneous Dirichlet problem for fractional-order
  quasi-linear operators with variable coefficients on Lipschitz domains $\Omega$ of $\R^d$.
  Our estimates are consistent with the boundary behavior of solutions
  on smooth domains and apply to fractional $p$-Laplacians and operators with finite horizon.
  The proof exploits the underlying variational structure
  and uses a new and flexible local translation operator. We further apply these regularity estimates to
  derive novel error estimates for finite element approximations of fractional $p$-Laplacians and
  present several simulations that reveal the boundary behavior of solutions.
\end{abstract}

\author[J.P.~Borthagaray]{Juan Pablo~Borthagaray}
\address[J.P.~Borthagaray]{Instituto de Matem\'atica y Estad\'istica ``Rafael Laguardia'', Facultad de Ingenier\'ia, Universidad de la Rep\'ublica, Montevideo, Uruguay}
\email{jpborthagaray@fing.edu.uy}
\thanks{JPB has been supported in part by Fondo Clemente Estable grant 2022-172393.}

\author[W.~Li]{Wenbo~Li}
\address[W.~Li]{Institute of Computational Mathematics and Scientific/Engineering Computing of the Chinese Academy of Sciences, Beijing 100190 China.}
\email{liwenbo@lsec.cc.ac.cn}

\author[R.H.~Nochetto]{Ricardo H.~Nochetto}
\address[R.H.~Nochetto]{Department of Mathematics and Institute for
Physical Science and Technology, University of Maryland, College
Park, MD 20742, USA}
\email{rhn@math.umd.edu}
\thanks{RHN has been supported in part by NSF grant DMS-1908267.}

\maketitle

\section{Introduction}
In recent years, fractional-order and, more generally, nonlocal operators have received a great deal of attention in applied sciences and engineering. This is mainly because such operators arise in jump processes modeling the ubiquitous phenomenon of anomalous diffusion \cite{MetzlerKlafter}. In this vein, 
the fractional Laplacian, an outstanding nonlocal operator, arises as a limit of a long-jump random walk \cite{Valdinoci}. Among other applications of nonlocal operators, we mention finance \cite{Rama:04, Levendorskii:04}, ground-water solute transport \cite{Benson:00}, and biological systems with binding, crowding, or trapping, such as electrodiffusion of ions within nerve cells \cite{Langlands:09,Langlands:11}.

For problems with a variational structure, finite element methods provide the best approximation in the energy norm, and are amenable to an analysis with low regularity conditions. In our setting, the latter is fundamental because solutions of fractional-order problems generically develop algebraic boundary layers. Solution regularity estimates in the Sobolev scale are a key ingredient to prove a priori convergence rates for the finite element discretization of such problems.  

However, most progress in that direction and most computational studies have been limited to either linear or semi-linear problems. This paper deals with fractional-order quasi-linear operators. We prove elliptic regularity estimates up to the boundary of the domain, which is only assumed to be bounded and Lipschitz. The model operator we consider is the so-called $(p,s)$-fractional Laplacian ($s \in (0,1)$, $p \in (1,\infty)$), but our theory is also valid for a broader class of operators, including operators with finite horizon.  In this regard, we remark that our regularity estimates for finite-horizon operators are even new for linear problems.
As an application of our regularity estimates, we consider direct finite element discretization of the problems under study and prove convergence rates in the energy norm.  

Let us make precise the problem setting in this paper. Let $\Omega \subset \Rd$ ($d \ge 1$) be a bounded, Lipschitz domain, $s \in (0,1)$, and $p \in (1,\infty)$. We consider energy functionals with domain the fractional-order Sobolev space $\widetilde{W}^s_p(\Omega)$, namely functions in $W^s_p(\R^d)$ that vanish in $\Omega^c:=\R^d\setminus\overline\Omega$. More precisely, for a given function $G \colon \R^d \times \R^d \times \R \to (0, \infty)$, with $(x,y,\rho) \mapsto G(x,y,\rho)$, and $f \in (\widetilde{W}^s_p(\Omega))'$, we are interested in minimizers of the energy
\begin{equation} \label{eq:def-energy}
\calF (u) := \iint_{Q_\Omega} G \left( x, y, \frac{u(x)-u(y)}{|x-y|^s} \right) \frac1{|x-y|^{d}} \, dy dx - \langle f, u \rangle.
\end{equation}
Above, $\langle \cdot, \cdot \rangle$ stands for the duality pairing between $(\widetilde{W}^s_p(\Omega))'$ and $\widetilde{W}^s_p(\Omega)$ and
\[
Q_\Omega := (\Rd \times \Rd) \setminus (\Omega^c \times \Omega^c).
\]
Specific requirements on $G$ are listed in Hypothesis \ref{hyp:G} below. The Gateaux differential of $\calF$ at $u$ is given by $\calA \colon \widetilde{W}^s_p(\Omega) \to (\widetilde{W}^s_p(\Omega))'$,
\begin{equation}\label{eq:strong}
\calA u (x) := \int_{\Rd}  \left[ G_\rho \left( x, y, \frac{u(x)-u(y)}{|x-y|^s} \right) - G_\rho \left( y, x, \frac{u(y)-u(x)}{|x-y|^s} \right) \right] \frac1{|x-y|^{d +s}} \, dy .
\end{equation}
For the moment, let us assume that $G$ satisfies the relation $G(x,y,\rho) = G(y,x,-\rho)$ for a.e. $x,y,\rho$. While this assumption allows us to write the minimization problem in a strong form in a concise fashion, it is not necessary for our theoretical results.
Under this additional condition, we have $ G_\rho(x,y,\rho) = -  G_\rho(y,x,-\rho)$ for a.e. $x,y,\rho$ and we can write 
\begin{equation}\label{eq:strong-symmetric}
\calA u (x) := 2 \int_{\Rd}  G_\rho \left( x, y, \frac{u(x)-u(y)}{|x-y|^s} \right) \frac1{|x-y|^{d +s}} \, dy .
\end{equation}

Minimizers of \eqref{eq:def-energy} are weak solutions of the homogeneous Dirichlet problem for the operator $\calA$:
\begin{equation} \label{eq:Dirichlet}
\left\lbrace \begin{array}{rl}
\calA u = f & \mbox{in } \Omega, \\
u = 0 & \mbox{in } \Omega^c .
\end{array} \right.
\end{equation}

We assume standard hypotheses on $G$ in order to apply the direct method in the calculus of variations. As a prototypical example, we consider $G(x,y,\rho) = \frac{C_{d,s,p}}{2p} |\rho|^p$ with $C_{d,s,p}$ defined below. Then, $G_\rho(x,y,\rho) = \frac{C_{d,s,p}}{2}|\rho|^{p-2} \rho$, and
\begin{equation} \label{eq:def-pLap}
\calA u (x) = \Lap u(x) := C_{d,s,p} \int_\Rd \frac{|u(x)-u(y)|^{p-2} (u(x)-u(y))}{|x-y|^{d+sp}} \, d y
\end{equation}
is the so-called {\em fractional $(p,s)$-Laplacian} (or {\em fractional $p$-Laplacian of order $s$}). We define the normalizing constant $C_{d,s,p}$ as
\begin{equation}\label{eq:Cdsp}
C_{d,s,p} = \frac{s(1-s) p \; \Gamma(\frac{ps + d}{2}) \; 2^{2s-2}}{\pi^{\frac{d-1}{2}} \Gamma(\frac{(p-2)s + 3}{2}) \Gamma(2-s)}.
\end{equation}
This choice is somewhat arbitrary, but for $p=2$ it allows us to recover the {\em integral fractional Laplacian,} which is the pseudodifferential operator with symbol $|\xi|^{2s}$. Moreover,
for every smooth function $v \in C_c^{\infty}(\Rd)$ we have the asymptotic behaviors \cite{BourBrezMiro2001another,Mazya_BBM,delTeso:21}
\begin{equation} \label{eq:asymptotic-behaviors}
\lim_{s \to 0^+}  (-\Delta)^s_p v = |v|^{p-2} v, \quad \lim_{s \to 1^-}(-\Delta)^s_p v = -\nabla \cdot (|\nabla v|^{p-2} \nabla v).
\end{equation}
We also point out that the integral in \eqref{eq:def-pLap} needs to be understood in the principal value sense if $s \ge 1 - \frac1p$.

Another way to write the operator in \eqref{eq:def-pLap} is
\begin{equation} \label{eq:equivalent-pLap}
\Lap u(x) = 2 \int_\Rd \left( \frac{|u(x)-u(y)|}{|x-y|^s} \right)^{p-2} \frac{(u(x)-u(y))}{|x-y|^{d+2s}} \, d y,
\end{equation}
which suggests that, heuristically, one can understand the fractional $(p,s)$-Laplacian as a weighted fractional Laplacian of order $s$, with a weight $ \left( \frac{|u(x)-u(y)|}{|x-y|^s} \right)^{p-2}$. This is analogous to the local case, for which the $p$-Laplacian $(-\Delta)_p u := \mbox{div}(|\nabla u|^{p-2} \nabla u)$ can be regarded as a Laplacian with weight $|\nabla u|^{p-2}$. The Dirichlet problem for the local $p$-Laplacian arises in a number of models of physical processes, including non-Newtonian fluids \cite{Atkinson:74}, turbulent flows in porous media \cite{Diaz:94},  and global climate modeling \cite{Diaz:06}. We refer to \cite{Benedikt:18} for a historical account and other applications of this operator, and to \cite{Barrett93, Chow89, GlMa75} for its numerical treatment.

The representation \eqref{eq:equivalent-pLap} also shows that the operator \eqref{eq:def-pLap} corresponds to a degenerate diffusion if $p > 2$ and to a singular one if $p<2$. We refer to \cite{Caffarelli:12} for several motivations for considering nonlinear operators like \eqref{eq:def-pLap}, to \cite{Mosconi:16} for a thorough discussion about existence and regularity results for problems driven by the fractional $(p,s)$-Laplacian, and to \cite{delTeso:23}  for a monotone finite difference scheme with consistency error estimates for $C^4$ functions and applications to the Cauchy problem  for such an operator.

Depending on whether the resulting operator $\calA$ in \eqref{eq:strong} is degenerate or singular, our regularity estimates are somewhat different from one another. To make the point clear, let us focus on the case of the $(p,s)$-Laplacian, although we emphasize that our estimates are valid for more general operators (cf. \Cref{thm:max-regularity} below). We derive shift estimates in Besov norms of the form
\begin{align}
& \| u \|_{\dot B^{s+\frac{1}{p}}_{p,\infty}(\Omega)} \lesssim \| f \|_{B^{-s+\frac{1}{p'}}_{p',1}(\Omega)}^{\frac1{p-1}}, & \mbox{ if } p \ge 2, \label{eq:p>2}\\
& \| u \|_{\dot B^{s + \frac{1}{2}}_{p,\infty}(\Omega)} \lesssim \|f\|_{W^{-s}_{p'}(\Omega)}^{\frac{2-p}{p-1}} \| f \|_{B^{-s+\frac{1}{2}}_{p',1}(\Omega)} , & \mbox{ if } p \le 2, \label{eq:p<2}
\end{align}
where $p' = \frac{p}{p-1}$; these Besov estimates extend classical ones \cite[Theorems 2 and 2']{Savare98} to the fractional setting. To check optimality, we consider the prototypical $1d$-function $v(x) = x_+^s$ which mimics the boundary behavior of solutions of \eqref{eq:Dirichlet} for the operator \eqref{eq:def-pLap}. A simple calculation using second differences shows that $v \in B^{s+1/p}_{p,\infty}(\Omega)$ for all $p\in (1,\infty)$, which revals that \eqref{eq:p>2} is optimal while \eqref{eq:p<2} is suboptimal. Moreover, by a simple embedding argument, \eqref{eq:p>2} and \eqref{eq:p<2} give rise to Sobolev regularity estimates. Estimate \eqref{eq:p>2} for $p=2$ turns out to be consistent with well-known optimal regularity for solutions to the Dirichlet problem for the integral fractional Laplacian on smooth domains, cf. \cite{AbelsGrubb, Grubb, VishikEskin}. Importantly, our estimates are valid for Lipschitz domains and in that sense generalize the ones derived in \cite{Akagi:18,BoNo21} to a quasi-linear setting. Additionally, our estimates are valid under general conditions on the function $G$. In this vein, we point out to \cite{FernandezBonder:22} where, for a class of nonlinear operators related to the ones in this work, analysis is performed in fractional-order Orlicz-Sobolev spaces and H\"older regularity estimates are derived for Dirichlet problems on bounded $C^{1,1}$ domains.

The paper is organized as follows. \Cref{sec:notation} collects preliminary material about function spaces and Lipschitz domains, introduces a flexible local translation operator that plays an instrumental role in our derivation of regularity estimates, specifies the assumptions we require on the energy, and discusses the use of localized translations in the proof of regularity of energy minimizers. \Cref{sec:regularity} contains the core of the paper, and studies the regularity of solutions through the derivation of suitable energy bounds. It also discusses the extension of the technique to operators with finite horizon and {\em truncated Laplacians} in the linear setting. \Cref{sec:FE} proposes and analyzes a finite element discretization of problems of the form \eqref{eq:Dirichlet}, and exploits \eqref{eq:p>2} and \eqref{eq:p<2} to prove error bounds for all $p\in (1,\infty)$. Finally, \Cref{sec:numerical} exhibits some numerical experiments that explore the accuracy of this approach and the boundary behavior of solutions to the $(p,s)$-Laplacian \eqref{eq:def-pLap} and linear truncated Laplacians.

\section{Notation and assumptions} \label{sec:notation}

This section establishes the notation and collects some preliminary results. 
We provide some discussion on function spaces and Lipschitz domains. We analyze function space characterizations by means of translation operators, discuss the relation between these translations and the regularity of minimizers of \eqref{eq:def-energy}, and introduce a suitable localized translation operator to derive regularity estimates. Finally, we make explicit assumptions on the energy, discuss some of their consequences, and comment on how they apply to the model operator \eqref{eq:def-pLap}.

\subsection{Sobolev and Besov spaces}
Here, we briefly review some important facts about Sobolev and Besov spaces. We follow the notation from \cite{BoNo21} and refer to that work for further details.

Given $\sigma \in (0,1)$ and $p \in [1,\infty)$, we consider the zero-extension Sobolev space 
\[
\widetilde{W}^\sigma_p(\Omega) := \big\{ v \in W^\sigma_p(\mathbb{R}^d) \colon \supp v \subset \overline{\Omega} \big\};
\]
this is a Banach space furnished with the norm
\begin{equation*} \label{eq:Gagliardo-seminorm}
\| v \|_{\widetilde{W}^\sigma_p(\Omega)} := |v|_{W^\sigma_p(\Rd)} =  \left(\frac{C_{d,s,p}}{2} \iint_{\Rd \times \Rd} \frac{|v(x)-v(y)|^p}{|x-y|^{d+\sigma p}} \, dx \, dy \right)^{1/p}.
\end{equation*}
Because functions in $\widetilde{W}^\sigma_p(\Omega)$ vanish in $\Omega^c$, the integrand above vanishes on $\Omega^c \times \Omega^c$ and one can effectively compute the integral over $Q_\Omega = (\Rd \times \Rd) \setminus (\Omega^c \times \Omega^c)$. 

We define Besov spaces by real interpolation. Given a pair of compatible Banach spaces $(X_0,X_1)$, $u \in X_0 + X_1$, $t>0$, and $p \in [1,\infty)$, we consider the $K$-functional
\begin{equation}\label{eq:K-functional}
K(t, u) := \inf \left\{ \| u_0 \|_{X_0} + t \| u_1 \|_{X_1} \colon u = u_0 + u_1, \ u_0 \in X_0, \ u_1 \in X_1 \right\}. 
\end{equation}
For $\theta \in (0,1)$ and $q \in [1, \infty]$, we define the interpolation spaces
\[
\big[X_0, X_1\big]_{\theta, q} := \{ u \in X_0 + X_1 \colon \| u \|_{(X_0, X_1)_{\theta, q}} < \infty \},
\]
where
\begin{equation}\label{eq:Besov-norm}
\| u \|_{[X_0, X_1]_{\theta, q}} :=
\begin{cases}
\Big[ q \theta (1-\theta) \int_0^\infty t^{-(1+\theta q)} |K(t,u)|^q \, dt \Big]^{1/q} & \mbox{if } 1 \le q < \infty, \\
\sup_{t > 0} \ t^{-\theta} |K(t,u)|  & \mbox{if } q = \infty.
\end{cases}
\end{equation}
The normalization factor $q \theta (1-\theta)$ in \eqref{eq:Besov-norm} guarantees the correct scalings in the limits $\theta \to 0$, $\theta \to 1$ and $q \to \infty$ for norm continuity; see \cite[Appendix B]{mclean2000strongly} for a detailed proof in the case $p=2$. Because we are interested in spaces with differentiability order between zero and two, we let $X_0:=L^p(\Omega)$, $X_1:=W^2_p(\Omega)$, $\sigma\in(0,2)$ and $q\in [1,\infty]$ to define the Besov spaces
\[
B^\sigma_{p,q} (\Omega) := \big[L^p(\Omega), W^2_p(\Omega)\big]_{\sigma/2,q},\quad
\dot{B}^\sigma_{p,q} (\Omega) := \{v\in B^\sigma_{p,q} (\Omega): \supp v \subset\overline{\Omega}\}.
\]
By reiteration, we have the following result regarding interpolation of Besov spaces (cf. \cite[Theorem 6.4.5]{BerghLofstrom}): given $\sigma_0 \ne \sigma_1$, $1 \le p, q_0, q_1, r \le \infty$ and $0 < \theta < 1$,
\begin{equation} \label{eq:interpolation_Besov}
\left( B^{\sigma_0}_{p,q_0}(\Omega), B^{\sigma_1}_{p,q_1}(\Omega) \right)_{\theta, r} = B^{\sigma}_{p,r}(\Omega) , \quad \mbox{where } \sigma = (1-\theta) \sigma_0 + \theta \sigma_1 .
\end{equation}
Importantly, we have $B^\sigma_{p,p} (\Omega) = W^\sigma_p(\Omega)$ for all $p \in [1,\infty)$, $\sigma \in (0,2) \setminus \{1\}$. In the case $\sigma = 1$, we only have the equality $B^1_{2,2} (\Omega) = H^\sigma(\Omega)$, while $B^1_{p,p} (\Omega) \subset W^1_p(\Omega)$ if $p<2$ and $B^1_{p,p} (\Omega) \supset W^1_p(\Omega)$ if $p >2$, cf. \cite[\S7.67]{AdamsFournier:2003}.
Moreover, we have the following inclusions between Besov spaces on bounded Lipschitz domains \cite[\S3.2.4, \S3.3.1]{Triebel10}: 
\[ \begin{array}{llll}
B^\sigma_{p,q_0}(\Omega) \subset B^\sigma_{p,q_1}(\Omega),  & \mbox{ if } \sigma > 0, &  1 \le p \le \infty, & 1 \le q_0 \le q_1 \le \infty; \\
B^{\sigma_1}_{p,q_1}(\Omega) \subset B^{\sigma_0}_{p,q_0}(\Omega) & \mbox{ if } 0 < \sigma_0 < \sigma_1, &   1 \le p \le \infty , & 1 \le q_0, q_1 \le \infty.
\end{array} \]
We are interested in making precise the statement about inclusion of a higher-order Besov space with integrability index $p \in [1,\infty)$ and second parameter $q=\infty$ into a lower-order Sobolev space with the same integrability index. Concretely, the next lemma shows the scaling of the continuity constant.

\begin{lemma}[embedding] \label{lemma:Besov-Sobolev-emb}
Let $\Omega \subset \Rd$ be a bounded Lipschitz domain, $p \in [1,\infty)$, $\sigma\in (0,2) \setminus \{1\}$, and $\eps \in (0,2-\sigma)$. Then, $B^{\sigma+\eps}_{p,\infty}(\Omega) \subset W^{\sigma}_p(\Omega)$ with 
\begin{equation}\label{eq:Besov-Sobolev-emb}
  \|v\|_{W^{\sigma}_p(\Omega)} \lesssim \left( \frac{\sigma (2-\sigma)}{\eps} \right)^{\frac1p} \|v\|_{B^{\sigma+\eps}_{p,\infty}(\Omega)}
  \quad\forall \, v\in B^{\sigma+\eps}_{p,\infty}(\Omega).
\end{equation}
\end{lemma}
\begin{proof}
We exploit the characterization of Besov and fractional-order Sobolev spaces as interpolation spaces between integer-order Sobolev spaces. More precisely, if the $K$-functional corresponds to interpolation between the spaces $X_0 = L^p(\Omega)$ and $X_1 = W^2_p(\Omega)$, we recall the norm definitions
\[
\| v \|_{B^{\sigma+\eps}_{p,\infty}(\Omega)} = \sup_{t > 0} \,\Big( t^{-\frac{\sigma+\eps}2} |K(t,v)|\Big),
\]
and
\[
\| v \|_{W^{\sigma}_p(\Omega)}^p =  \frac{p \sigma (2-\sigma)}4 \int_0^\infty  t^{-1-\frac{\sigma p}2} |K(t,v)|^p dt
\]
for $\sigma\in(0,2)\setminus\{1\}$, according to the remark following \eqref{eq:interpolation_Besov}.
We split the integral above as the sum of the integrals between $0$ and $N$ and between $N$ and $\infty$, with $N>0$ to be chosen. A straightforward calculation gives
\[
\int_0^N  t^{-1-\frac{\sigma p}2} |K(t,v)|^p dt \le \sup_{t > 0} \, t^{-\frac{(\sigma+\eps)p}2} |K(t,v)|^p \,\int_0^N t^{-1+\frac{\eps p}2} dt = \frac{2 N^{\frac{\eps p}2}}{\eps p} \| v \|_{B^{\sigma+\eps}_{p,\infty}(\Omega)}^p.
\]
Additionally, for any $v\in B^{\sigma+\eps}_{p,\infty}(\Omega)$ and $t \ge 0$, we choose the trivial decomposition $v = v + 0$ in \eqref{eq:K-functional} to obtain
\[
|K(t,v)| \le \| v \|_{L^p(\Omega)} \le \| v \|_{W^{\sigma}_p(\Omega)}.
\]
This gives rise to
\[
\int_N^\infty  t^{-1-\frac{\sigma p}2} |K(t,v)|^p dt \le \| v \|_{W^{\sigma}_p(\Omega)}^p \int_N^\infty  t^{-1-\frac{\sigma p}2} dt = \frac{2 N^{-\frac{\sigma p}2}}{p \sigma} \| v \|_{W^{\sigma}_p(\Omega)}^p ,
\]
and thus
\[
\| v \|_{W^{\sigma}_p(\Omega)}^p \le \frac{\sigma (2-\sigma) N^{\frac{\eps p}2}}{2 \eps}
 \| v \|_{B^{\sigma+\eps}_{p,\infty}(\Omega)}^p + 
\frac{(2-\sigma)N^{-\frac{\sigma p}2}}{2} \| v \|_{W^{\sigma}_p(\Omega)}^p.
\]
It now suffices to fix $N$ sufficiently large so that $\frac{(2-\sigma)N^{-\frac{\sigma p}2}}{2} < 1$ and kick back the last term in the right hand side above to arrive to the desired estimate \eqref{eq:Besov-Sobolev-emb}.
\end{proof}

For $\sigma \in (0,1)$ and $p,q \in [1, \infty]$, we define
\begin{equation*} \label{eq:def-negative-Besov}
B^{-\sigma}_{p,q} (\Omega) := \big(L^p(\Omega), W^{-1}_p(\Omega) \big)_{\sigma,q},
\end{equation*}
and point out that, if $p,q \in (1, \infty]$ and $p',q'\in[1,\infty)$ are the conjugate exponents,
we have the duality \cite{BoNo21}
\[
\dot{B}^\sigma_{p,q}(\Omega) = (B^{-\sigma}_{p',q'} (\Omega))'.
\]

It is common practice to furnish Besov spaces with equivalent norms based on $L^p$-norms of difference quotients,  instead of the interpolation norm. Given $\rho>0$,
\[
\Omega_\rho := \{x\in\Omega: \textrm{dist} (x,\partial\Omega)<\rho\},
\quad
\Omega^\rho := \{x\in\R^d: \textrm{dist} (x,\partial\Omega)>\rho\},
\]
and a set of admissible directions $D \subset \Rd$, typically a ball, we denote
\[  
|v|_{B^\sigma_{p,q}(\Omega; D)} := \Big(q\sigma(2-\sigma)
\int_D \frac{\| v_h - 2 v + v_{-h} \|_{L^p(\Omega_{|h|})}^q}{|h|^{d+q\sigma}}dh \Big)^{1/q}
\]
 for $p,q\in[1,\infty)$ while for $q=\infty$ we let
\[
|v|_{B^\sigma_{p,\infty}(\Omega; D)} := \sup_{h \in D} 
\frac{\| v_h - 2 v + v_{-h} \|_{L^p(\Omega_{|h|})}}{|h|^{\sigma}},
\]
where $v_h(x) := v(x+h)$ is the translation with vector $h \in \Rd$. It is well-known that, if $D$ is a ball, then the norm $\| \cdot \|_{L^p(\Omega)} + |\cdot|_{B^\sigma_{p,q}(\Omega; D)}$ is equivalent to the Besov norm $\| \cdot \|_{B^\sigma_{p,q}(\Omega)}$ defined through interpolation \cite[Theorem 7.47]{AdamsFournier:2003}. Moreover, \cite[Proposition 2.2]{BoNo21} shows that balls $D$ can be replaced by suitable convex cones in the definition of Besov seminorms for $q=\infty$.
More precisely, let us assume $D \subset \Rd$ is bounded and star-shaped with respect to the origin. We say that 
$D$ {\em generates} $\Rd$ if there exists $\rho_0(D) > 0$ such that for every $\rho \le \rho_0(D)$ and
every $h\in D_\rho(0)$, the ball of radius $\rho$ and center $0$, there exists $\{h_j\}_{j=1}^d \subset D\cup(-D)$ satisfying
\[
h = \sum_{j=1}^d h_j,
\quad
\sum_{j=1}^d |h_j| \le c |h|
\]
with a constant $c>0$ only dependent on $D$.
The following equivalence is proved in \cite[Proposition 2.2]{BoNo21}.

\begin{proposition}[Besov seminorms using cones]\label{P:seminorms}
Let $D$ be a convex cone generating $\Rd$ and let $B \subset \Rd$ be a ball. If $\sigma\in(0,2)$ and $p\in[1,\infty)$, then for every function $v \colon \Rd \to \R$ we have $|v|_{B^\sigma_{p,\infty}(\Omega; D)} \simeq |v|_{B^\sigma_{p,\infty}(\Omega;B)}$.
\end{proposition}

We will decompose $\Omega$ into overlapping balls and apply Proposition \ref{P:seminorms} to subdomains $\omega$ made of intersections of such balls with $\Omega$. The convex cone $D$ will depend on $\omega$ and will be dictated by the Lipschitz property of $\Omega$. However, we will omit writing $D$ in $B^\sigma_{p,q}(\omega)$ for simplicity of notation but without compromising clarity.

We can estimate higher-order Besov norms, possibly of order higher than one, in terms of difference quotients of Besov norms of order less than one. We express this instrumental reiteration property as follows and refer to \cite[Proposition 2.1]{BoNo21}.

\begin{proposition}[reiteration of Besov seminorms] \label{prop:bound-Besov}
Let $\omega \subset \Rd$ be a bounded Lipschitz domain, $s \in (0,1)$, $p,q \in [1, \infty]$, $\sigma \in (0, 1]$ and let $D$ be a set generating $\Rd$ and star-shaped with respect to the origin. Then, 
\[ \begin{split}
& |v|_{B^{s+\sigma}_{p,q}(\omega)} \lesssim \left( \int_{D} \frac{|v - v_h|^q_{W^s_p(\omega)}}{|h|^{d + q\sigma}}  \, dh \right)^{1/q}, \quad q \in [1,\infty), \\ 
& |v|_{B^{s+\sigma}_{p,\infty}(\omega)} \lesssim  \sup_{h \in D} \frac{1}{|h|^{\sigma}} |v - v_h|_{W^s_p(\omega)}. 
\end{split}\]
\end{proposition}

\subsection{Lipschitz domains}
We next briefly state a few well-known but relevant results regarding Lipschitz domains in $\Rd$.

\begin{definition}[admissible outward vectors] \label{def:outward_vectors}
For every $x_0 \in \Rd$ and $\rho \in (0,1]$, we define the set of admissible outward vectors
\[
\calO_\rho(x_0) = \{ h \in \Rd \colon |h|\le \rho, (B_{2\rho}(x_0)\setminus\Omega) + th \subset \Omega^c, \ \forall t \in [0,1] \}.
\]
\end{definition}

An important fact about bounded Lipschitz domains is that they satisfy a uniform cone property. This can be stated in the following fashion \cite[\S 1.2.2]{Grisvard}.

\begin{proposition}[uniform cone property] \label{prop:cone-property}
If $\Omega$ is a bounded Lipschitz domain, then there exist $\rho \in (0,1]$, $\theta \in (0, \pi]$ and a map $\vn \colon \Rd \to S^{d-1}$ such that, for every $x \in \Rd$,
\[
\calC_\rho (\vn(x), \theta) := \{ h \in \Rd \colon  |h| \le \rho, \ h \cdot \vn \ge |h| \cos \theta \} 
\subset \calO_\rho(x) .
\]
\end{proposition}

Besov seminorms can be equivalently written as sums of norms over partitions, as long as the partitions have some overlap. We refer to \cite[Lemma 2.6]{BoNo21}.

\begin{lemma}[localization]\label{lem:localization}
    Let $p, q \in [1, \infty]$ and $\sigma \in (0,2)$. Let $\{ D_j \}_{j = 1}^J$ be a finite covering of $\Omega$ by balls of radius $\rho$, $D_j = D_\rho(x_j)$. Then, $v \in B^\sigma_{p,q}(\Omega)$ if and only if $v \big|_{\Omega \cap B_j} \in B^\sigma_{p,q}(\Omega \cap D_j)$ for all $j = 1, \ldots, J$, and
\begin{equation} \label{eq:localization1}  
\| v \|_{B^\sigma_{p,q}(\Omega)}^p\simeq \sum_{j=1}^J \| v \|_{B^\sigma_{p,q}(\Omega \cap D_j)}^p.
\end{equation}
Moreover, for $\delta \ge \rho$, let $\{ D_j \}_{j = 1}^J$ be a finite covering of $\Omega^\delta$ 
and let $v \colon \Rd \to \mathbb{R}$ be such that $\mbox{supp}(v) \subset \overline\Omega$. Then, $v \in \dot{B}^\sigma_{p,q}(\Omega)$ if and only if $v \big|_{D_j} \in B^\sigma_{p,q}(D_j)$ for all $j = 1, \ldots, J$, and
\begin{equation} \label{eq:localization2}
\| v \|_{\dot{B}^\sigma_{p,q}(\Omega)}^p \simeq | v |_{\dot{B}^\sigma_{p,q}(\Omega)}^p\simeq \sum_{j=1}^J | v |_{B^\sigma_{p,q}(D_j)}^p.
\end{equation}
The equivalence constants above depend on $s,p,q,\Omega$ and the covering chosen.
\end{lemma}

\subsection{Localized translation operator}
Our next goal is to construct a smooth operator that resembles a translation around a certain given point $x_0\in\Rd$, while coincides with the identity away from $x_0$. Such {\em localized translation operator} plays an instrumental role in our derivation of regularity estimates. 

Given $x_0$ and $\rho$, we fix a cut-off function $\phi$ such that $0 \le \phi \le 1$, $\phi \equiv 1$ on the ball $D_\rho (x_0)$ of radius $\rho$ centered at $x_0$, $\mbox{supp}(\phi) \subset B_{2\rho}(x_0)$. Given $h \in \Rd$, we define
\begin{equation} \label{eq:translation-operator}
T_h v(x) := v\big(x + h \phi(x)\big) = \big( v \circ S_h \big) (x),
\end{equation}
where the map $S_h := I + h \phi$ is defined from $\Rd$ to $\Rd$. We restrict our consideration to $|h|$ small enough such that the Jacobian of $S_h$ satisfies
\[
\frac12 I \preccurlyeq \nabla S_h = I + h \otimes \nabla \phi \preccurlyeq 2 I.
\]
and thus $S_h$ is a one-to-one mapping from $D_{2\rho}(x_0)$ to $D_{2\rho}(x_0)$. It is also one-to-one from $\Rd$ to $\Rd$ and coincides with the identity in $D_{2\rho}(x_0)^c$.

In \cite{Savare98,BoNo21}, the localized translation operator
$ \widetilde{T}_hv := \phi v_h + (1-\phi) v $
was employed instead. The translation operator $T_h$ in \eqref{eq:translation-operator} is somewhat more flexible than $\widetilde{T}_h$, in the sense that it gives rise to cleaner regularity estimates in which a priori one gains one full derivative; compare the right-hand side in \eqref{eq:regularity-FG} below with the one in \cite[formula (3.3)]{BoNo21}. This leads to a simpler bootstrapping argument than in \cite{BoNo21}. We return to this point in \Cref{rmk:translation-operator} below.

\begin{remark}[properties of $S_h$ and $S_h^{-1}$] \label{rmk:Sh}
  Some important properties of the transformations $S_h$ and $S_h^{-1}$ follow immediately from their definitions. We have the inequalities
\begin{gather}\label{eq:easy-Sh}
|S_h(x) - x | \lesssim |h|, \quad \forall x \in \Rd,
\\ \label{eq:jacobian-Sh}
\left| \det(\nabla S_h(x)) - 1 \right|\lesssim |h|, \quad \forall x \in \Rd,
\\ \label{eq:quasi-isometry-Sh}
\left| \frac{|x-y|}{| S_h(x) - S_h(y) |} - 1 \right| \lesssim |h|, \quad \forall x,y \in \Rd.
\end{gather}
Analogous properties also hold for $S^{-1}_h$.
\end{remark}

\begin{remark}[boundedness of translations] Clearly, the operator $T_h$ in \eqref{eq:translation-operator}  is bounded from $L^p(\Rd)$ to $L^p(\Rd)$ and, more in general, from $W^k_p(\Rd)$ to $W^k_p(\Rd)$ for every $k \in \N$. Therefore, it is also bounded from $B^\sigma_{p,q}(\Rd)$ to $B^\sigma_{p,q}(\Rd)$ for any non-integer $\sigma > 0$. Moreover, if $h \in \calO_\rho(x_0)$ is an admissible outward vector (cf. Definition \ref{def:outward_vectors}) and $v \in \dot{B}^\sigma_{p,q}(\Omega)$, we have
\[ \begin{aligned}
& x \in \Omega^c \cap D_{2\rho}(x_0): \quad 0 \le \phi(x) \le 1 \Rightarrow S_h(x) = x + h \phi(x) \in \Omega^c \Rightarrow T_hv(x) = 0, \\
& x \in \Omega^c \setminus D_{2\rho}(x_0): \quad \phi(x) = 0 \Rightarrow T_hv(x) = v(x) = 0.
\end{aligned} \]
Therefore, $T_h$ is also a bounded operator from $\dot{B}^\sigma_{p,q}(\Omega)$ to $\dot{B}^\sigma_{p,q}(\Omega)$.
\end{remark}

\begin{lemma}[moduli of continuity]\label{lemma:translation}
Let $p, q \in [1, \infty]$. For all $\sigma \in (0,1)$ we have
\begin{equation} \label{eq:localized-translation-regularity}
\| v - T_h v \|_{L^p(D_{2\rho}(x_0))} \lesssim |h|^\sigma \|v\|_{B^{\sigma}_{p,q}(D_{2\rho}(x_0))} \quad \forall v \in B^{\sigma}_{p,q}(D_{2\rho}(x_0)).
\end{equation}
Moreover, for all $r > 0$ and $\sigma \in [0, 1]$, we have
\begin{equation} \label{eq:localized-translation-regularity-higher}
\| v - T_h v \|_{B^{r}_{p,q}(D_{2\rho}(x_0))} \lesssim |h|^\sigma \|v\|_{B^{r+\sigma}_{p,q}(D_{2\rho}(x_0))} \quad \forall v \in B^{r+\sigma}_{p,q}(D_{2\rho}(x_0)).
\end{equation}
\end{lemma}

\begin{proof}
In first place, we remark that because $T_h v = v \circ S_h$,  $T_hv -v$ vanishes on $D_{2\rho}(x_0)^c$. The boundedness of $T_h$ on $L^p(\Rd)$ yields
\[
\| v - T_h v \|_{L^p(D_{2\rho}(x_0))} \lesssim \| v \|_{L^p(D_{2\rho}(x_0))},
\]
while a standard calculation gives
\begin{equation} \label{eq:localized-translation-W1p}
\| v - T_h v \|_{L^p(D_{2\rho}(x_0))} \lesssim |h| \, \| \nabla v \|_{L^p(D_{2\rho}(x_0))}.
\end{equation}
Estimate \eqref{eq:localized-translation-regularity} follows by interpolation.

We next consider higher-order derivatives. Given a positive integer $k$ and $\alpha$ with $|\alpha| = k$, a direct calculation gives for any sufficiently smooth $v$ and $x \in D_{2\rho}(x_0)$,
\[
\left| \left(D^{\alpha} v \right)  \circ S_h (x) - D^{\alpha} \left( v \circ S_h(x) \right) \right| \lesssim  |h| \Vert \phi \Vert_{C^{k}(\Rd)} \sum_{0 < |\alpha'| \le k} \big| \big(D^{\alpha'} v\big) \circ S_h(x) \big|.
\]
Consequently, applying \eqref{eq:localized-translation-W1p} to $\left(D^{\alpha} v \right) \circ S_h$, we deduce
\[
\begin{aligned}
&\| v - T_h v \|_{W^k_p(D_{2\rho}(x_0))} = \| v - v \circ S_h \|_{W^k_p(D_{2\rho}(x_0))} \\
&\le \!\!\!\! \sum_{0 \le |\alpha| \le k} \!\!\!\! \| \! \left(D^{\alpha} v \right) \!\circ\! S_h \! - \! D^{\alpha} \left( v \!\circ\! S_h \right) \! \|_{L^p(D_{2\rho}(x_0))} \! + \! \| \! \left(D^{\alpha} v \right) \!\circ\! S_h \! - \! D^{\alpha} v  \|_{L^p(D_{2\rho}(x_0))} \\
&\lesssim |h| \!\! \sum_{0 < |\alpha'| \le k+1} \| D^{\alpha'} v \|_{L^p(D_{2\rho}(x_0))}
\le |h| \| v \|_{W^{k+1}_p(D_{2\rho}(x_0))}.
\end{aligned}
\]
By using the boundedness of $T_h$ on $W^k_p(\Rd)$ and interpolation, we deduce that
\begin{equation*}
\| v - T_h v \|_{W^k_p(D_{2\rho}(x_0))} \lesssim |h|^\sigma \|v\|_{B^{k+\sigma}_{p,q}(D_{2\rho}(x_0))} \quad \forall v \in B^{k+\sigma}_{p,q}(D_{2\rho}(x_0))
\end{equation*}
for $\sigma \in (0, 1)$. Finally, we obtain \eqref{eq:localized-translation-regularity-higher} by combining this estimate with \eqref{eq:interpolation_Besov}.
\end{proof}

\begin{remark}
The only properties of $S_h$ we exploited in the previous lemma are the fact that it maps $D_{2\rho}(x_0)$ onto $D_{2\rho}(x_0)$, and that the resulting translation operator $T_h v = v \circ S_h$ is stable in $W^k_p(D_{2\rho}(x_0))$ and satisfies properties like \eqref{eq:localized-translation-W1p}. Thus, the same arguments can be applied to the translation operator $v \circ S_h^{-1}$. In particular, we have 
\[
\| v - v\circ S_h^{-1} \|_{W^k_p(D_{2\rho}(x_0))} \lesssim |h|^\sigma |v|_{B^{k+\sigma}_{p,q}(D_{2\rho}(x_0))} \quad \forall v \in B^{k+\sigma}_{p,q}(D_{2\rho}(x_0))
\]
for $\sigma \in (0, 1)$, and 
\[
\| v - v\circ S_h^{-1} \|_{B^{r}_{p,q}(D_{2\rho}(x_0))} \lesssim |h|^\sigma \|v\|_{B^{r+\sigma}_{p,q}(D_{2\rho}(x_0))} \quad \forall v \in B^{r+\sigma}_{p,q}(D_{2\rho}(x_0))
\]
for $r > 0, \sigma \in [0, 1]$.
\end{remark}

\subsection{Assumptions on the energy}
We recall the energy \eqref{eq:def-energy},
\[
\calF (u) = \iint_{Q_\Omega} G \left( x, y, \frac{u(x)-u(y)}{|x-y|^s} \right) \frac1{|x-y|^{d}} \, dy dx - \langle f, u \rangle.
\]
Here, we list the conditions we require on this functional, discuss some consequences of these conditions, and how they apply to the problems we are interested in.

\begin{hypothesis} \label{hyp:G} The function $G \colon \R^d \times \R^d \times \R \to (0, \infty)$ satisfies the following conditions.

\smallskip

\noindent {\bf 1. Convexity.} The function $\rho \mapsto G(x,y,\rho)$ is uniformly convex for each $x,y \in \Rd$.

\smallskip

\noindent {\bf 2. Space continuity.} There exists $\beta \in (0,1]$ such that $G$ is $\beta$-H\"older continuous with respect to the space variables:
for all $x,y, x',y' \in \Rd$ and $\rho \in \R,$
\begin{equation} \label{eq:G-Lipschitz}
|G(x,y,\rho) - G(x',y',\rho)| \le C (|x-x'|^\beta + |y-y'|^\beta) |\rho|^p. 
\end{equation}

\smallskip

\noindent {\bf 3. $p$-growth.} There exists some $p \in (1,\infty)$ such that, for all $x,y \in \Rd$, $G(x,y,\rho)$ is differentiable with respect to $\rho$ with
\begin{equation}\label{eq:p-growth}
| G(x,y,\rho) | \le C |\rho|^p , \qquad | G_\rho(x,y,\rho) | \le C |\rho|^{p-1}.
\end{equation}

\smallskip

\noindent {\bf 4. Monotonicity.} There exists $\alpha > 0$ such that, if $2 \le p < \infty$ in the {\bf $p$-growth} condition, then for all 
$x,y, \in \Rd$ and $\rho, \rho' \in \R,$
\begin{equation*} \label{eq:G-monotone-p-big}
(G_\rho (x,y,\rho) - G_\rho (x,y,\rho')) (\rho - \rho') \ge \alpha |\rho - \rho'|^p ,
\end{equation*}
while if $1< p < 2$, then  for all  $x,y, \in \Rd$ and $\rho, \rho' \in \R,$
\begin{equation*} \label{eq:G-monotone-p-small}
(G_\rho (x,y,\rho) - G_\rho (x,y,\rho')) (\rho - \rho') \ge \alpha |\rho - \rho'|^2 \big| |\rho| + |\rho'| \big |^{p-2}.
\end{equation*}

\smallskip

\noindent {\bf 5. Continuity.}  There exists $c > 0$ such that, if $2 \le p < \infty$ in the {\bf $p$-growth} condition, then
\begin{equation*} \label{eq:G-cont-p-big}
|G_\rho (x,y,\rho) - G_\rho (x,y,\rho')| \le c |\rho - \rho'| \, \big| |\rho| + |\rho'| \big|^{p-2} \quad \forall x,y, \in \Rd, \ \rho, \rho' \in \R,
\end{equation*}
while if $1< p < 2$, then
\[
|G_\rho (x,y,\rho) - G_\rho (x,y,\rho')| \le c |\rho - \rho'|^{p-1} \quad \forall x,y, \in \Rd, \ \rho, \rho' \in \R.
\]
\end{hypothesis}

\begin{remark}[symmetry]
While not strictly needed for our purposes, the following assumption is practical for the analysis and applies to a general class of operators:

\smallskip

\noindent {\bf 6. Symmetry.} The function $G$ is symmetric with respect to the space variables and with respect to $\rho$,
\[
G(x,y,\rho) = G(y,x,\rho), \quad G(x,y,\rho) = G(x,y,-\rho), \quad \forall x,y \in \Rd, \rho \in \R.
\]

\smallskip

Under this symmetry assumption, the operator $\calA$ associated with the energy minimization problem becomes \eqref{eq:strong-symmetric}; otherwise, it takes the form \eqref{eq:strong}.
\end{remark}

\begin{remark}[monotonicity] \label{rem:monotonicity}
The {\bf monotonicity} hypothesis above implies the following estimates for the operator $\calA$ in \eqref{eq:strong}. If $2 \le p < \infty$, there exists $\alpha > 0$ such that,
for all $ u,v \in \widetilde W^s_p(\Omega),$ 
\begin{equation}\label{eq:monotonicity-pge2}
\langle \calA u - \calA v, u - v \rangle \ge \alpha \| u-v \|_{\widetilde{W}^s_p(\Omega)}^p;
\end{equation}
hence $\calA$ is $p$-coercive in $\widetilde W^s_p(\Omega)$. Instead, if $1< p < 2$, we have for all $ u,v \in \widetilde W^s_p(\Omega)$
\begin{equation}\label{eq:monotonicity-ple2}
\langle \calA u - \calA v, u - v \rangle \ge \alpha \| u-v \|_{\widetilde{W}^s_p(\Omega)}^2 \left( \| u \|_{\widetilde{W}^s_p(\Omega)} + \| v \|_{\widetilde{W}^s_p(\Omega)} \right)^{p-2},
\end{equation}
whence $\calA$ is $2$-coercive on bounded sets in $\widetilde W^s_p(\Omega)$:
\[
\langle \calA u - \calA v, u - v \rangle \ge C(R) \| u - v\|_{\widetilde{W}^s_p(\Omega)}^2 
\]
provided $\| u \|_{\widetilde{W}^s_p(\Omega)}, \| v \|_{\widetilde{W}^s_p(\Omega)} \le R$.
It is worth realizing that \eqref{eq:monotonicity-pge2} cannot hold for $p<2$ and $G$ smooth and convex
  \cite[Remark 2.1]{Savare98}. This fact is responsible for the dichotomy between \eqref{eq:p>2} and \eqref{eq:p<2}
  and reveals that our variational approach, which hinges on \eqref{eq:monotonicity-ple2} for $p<2$, cannot improve upon \eqref{eq:p<2}.
\end{remark}

\begin{remark}[continuity] \label{rem:continuity}
The {\bf continuity} hypothesis implies that the operator $\calA$ satisfies the following bounds: if $p \in (1,2]$ then for all $u,v \in \widetilde{W}^s_p(\Omega)$,
\begin{equation} \label{eq:continuity-ple2}
\| \calA u - \calA v \|_{W^{-s}_{p'}(\Omega)} \le C \| u-v \|_{\widetilde{W}^s_p(\Omega)}^{p-1} ,
\end{equation}
while if $p \in [2,\infty)$ then for all $u,v \in \widetilde{W}^{s,p}(\Omega)$,
\begin{equation} \label{eq:continuity-pge2}
\| \calA u - \calA v \|_{W^{-s}_{p'}(\Omega)} \le C \left( \| u \|_{\widetilde{W}^s_p(\Omega)} + \| v \|_{\widetilde{W}^s_p(\Omega)} \right)^{p-2} \| u-v \|_{\widetilde{W}^s_p(\Omega)} .
\end{equation}
\end{remark}

\begin{remark}[solution operator] \label{R:solution-operator}
 The uniform convexity of $G$ yields existence and uniqueness of weak solutions: given $f \in W^{-s}_{p'}(\Omega)$, the
  problem
\begin{equation}\label{eq:weak}
  u_f \in \widetilde{W}^s_p(\Omega) \quad
  \langle \calA u_f , v \rangle = \langle f , v \rangle \quad \forall v \in  \widetilde{W}^s_p(\Omega),
\end{equation}    
admits a unique solution, where
\begin{equation*}  
  \langle \calA u , v \rangle = \iint_{\Rd\times\Rd} \widetilde G \left( x, y, \frac{u(x)-u(y)}{|x-y|^s} \right) \frac{(u(x)-u(y))(v(x)-v(y))}{|x-y|^{d+2s}} \, dx \, dy
\end{equation*}
and $\widetilde G (x,y, \rho) := G_\rho (x,y,\rho)/\rho$.     
Testing \eqref{eq:weak} with $v = u_f$ (or, equivalently, setting $v\equiv0$ in \eqref{eq:monotonicity-pge2}--\eqref{eq:monotonicity-ple2}), we immediately reach the stability estimate
\begin{equation} \label{eq:stability}
\| u_f \|_{\widetilde{W}^s_p(\Omega)} \le \frac1{\alpha^\frac1{p-1}}\| f \|_{W^{-s}_{p'}(\Omega)}^{\frac1{p-1}}.
\end{equation}
We now assess the continuity properties of the solution operator $f\mapsto u_f$.
If we assume $\| f \|_{W^{-s}_{p'}(\Omega)} \le K$, then $u_f$ satisfies $\|u_f\|_{\wt{W}^s_p(\Omega)} \le \left(\frac{K}\alpha\right)^{\frac{1}{p-1}}$ in view of \eqref{eq:stability}. This shows that, denoting
\[
\overline{B}_K := \{ f \in W^{-s}_{p'}(\Omega) \colon \| f \|_{W^{-s}_{p'}(\Omega)} \le K \},
\]
and using \eqref{eq:monotonicity-ple2},
then the solution operator defined on $\overline{B}_K$ is Lipschitz continuous,
\begin{equation} \label{eq:local-Lip-solution-op}
\|u_f - u_g \|_{\widetilde{W}^s_p(\Omega)} \le c(K) \| f - g \|_{W^{-s}_{p'}(\Omega)}
\end{equation}
for $p\in(1,2)$.
In contrast, if $p \in [2,\infty)$, then the solution map is H\"older continuous on $W^{-s}_{p'}(\Omega)$ because of \eqref{eq:monotonicity-pge2},
\begin{equation} \label{eq:Holder-solution-op}
\|u_f - u_g \|_{\widetilde{W}^s_p(\Omega)} \le \frac1{\alpha^\frac1{p-1}}\| f -g \|_{W^{-s}_{p'}(\Omega)}^{\frac1{p-1}}.
\end{equation}
\end{remark}

\bigskip\noindent
{\bf Fractional $(p,s)$-Laplacians.} For $p \in (1,\infty)$ we consider $G(x,y,\rho) = \frac{C_{d,s,p}}{2 p}|\rho|^p$ in \eqref{eq:def-energy}, which gives rise to \eqref{eq:def-pLap}. It is clear that the parameter $p$ in this definition corresponds to the parameter $p$ in Hypothesis \ref{hyp:G}, and therefore the {\bf convexity} and {\bf $p$-growth} conditions \eqref{eq:p-growth} hold. Moreover, because $G$ is independent of the space variables, it is trivially symmetric with respect to $x,y$, and we can take $\beta = 1$ in  \eqref{eq:G-Lipschitz}. 
The {\bf monotonicity} and {\bf continuity} assumptions are satisfied because of the following auxiliary identities \cite[\S5]{GlMa75}: 
for all $a,b \in \R$, we have
\begin{equation*} \label{eq:aux-ineq}
\left| |a|^{p-2} a - |b|^{p-2} b \right| \le 
\left\lbrace \begin{array}{rl}
C |a-b|^{p-1} & \mbox{if } 1 < p \le 2, \\
C |a-b| (|a|+|b|)^{p-2} & \mbox{if } 2 \le p < \infty,
\end{array} \right.
\end{equation*}
and
\begin{equation*} \label{eq:aux-ineq2}
\left( |a|^{p-2} a - |b|^{p-2} b \right) (a - b) \ge 
\left\lbrace \begin{array}{rl}
\alpha |a-b|^2 (|a|+|b|)^{p-2} & \mbox{if } 1 < p \le 2, \\
\alpha  |a-b|^p & \mbox{if } 2 \le p < \infty.
\end{array} \right.
\end{equation*}
Therefore, the Hypothesis \ref{hyp:G} covers fractional $(p,s)$-Laplace operators \eqref{eq:def-pLap}. 

\subsection{Regularity of functionals}
Inspired by \cite{Savare98}, we introduce a notion of regularity of functionals that measures their sensitivity  with respect to a family of perturbations.

\begin{definition}[$(T,D,\sigma)$-regularity]\label{def:regularity}
Let $V$ be a Banach space, $K \subset V$ and $\sigma >0$. Given a family of maps $T_h \colon K \to K$, with $h$ varying on a given set $D \subset \Rd,$ we say that a functional $\calF$ is $(T,D, \sigma)$-regular on $K$ if, for all $v \in K$,
\[
\omega(v) = \omega(v; \calF, T, D, \sigma) := \sup_{h \in D} \frac{\calF (T_h v) - \calF (v)}{|h|^\sigma} < \infty .
\]
\end{definition}

\begin{remark}[subadditivity]
The modulus $\omega$ of $(T,D,\sigma)$-regularity is subadditive with respect to the $\calF$-argument:
\begin{equation} \label{eq:subadditivity}
\omega(v; \calF_1 + \calF_2, T, D, \sigma) \le \omega(v; \calF_1, T, D, \sigma)
+ \omega(v; \calF_2, T, D, \sigma).
\end{equation}
Thus, in order to prove the $(T,D,\sigma)$-regularity of $\calF_1 + \calF_2$, it suffices to show the regularity of each of the two functionals separately.
\end{remark}

A key consequence of the {\bf monotonicity} assumption in Hypothesis \ref{hyp:G} is the following estimate \cite[Theorem 1 and Corollary 1]{Savare98}.

\begin{lemma}[regularity and minimizers] \label{lem:corollary1}
Let $x_0 \in \Rd$, $\rho > 0$ and $h \in \mathcal{O}_\rho(x_0)$. Consider translation operators $T_h \colon \widetilde{W}^s_p(\Omega) \to \widetilde{W}^s_p(\Omega)$ as in \eqref{eq:translation-operator}.
If $u$ solves \eqref{eq:Dirichlet}, the functional $\calF$ defined in \eqref{eq:def-energy} satisfies Hypotheses \ref{hyp:G} and it is $(T,D,\sigma)$-regular on $\widetilde{W}^s_p(\Omega)$ for some $\sigma >0$, then the following hold:
\begin{itemize}
\item If $p\ge2$, then
\begin{equation} \label{eq:reg-min-p-big}
\alpha \|u-T_h u\|_{\widetilde{W}^s_p(\Omega)}^p \le p \omega(u) |h|^\sigma.
\end{equation}
\item If $1<p<2$, then 
\begin{equation} \label{eq:reg-min-p-small}
\alpha \|u-T_h u\|_{\widetilde{W}^s_p(\Omega)}^2 \le C(p) \left( \|u\|_{\widetilde{W}^s_p(\Omega)} + \|T_h u\|_{\widetilde{W}^s_p(\Omega)}\right)^{2-p} \omega(u) |h|^\sigma.
\end{equation}
\end{itemize}
\end{lemma}

Let us explain the crucial role of Lemma \ref{lem:corollary1} in the proof of regularity of solutions, and how localized translations come into play. Let $D_\rho(x_0)$ be a ball with center $x_0 \in \Omega$ and radius $\rho$ satisfying \Cref{prop:cone-property} (uniform cone property), and let $D = \calC_\rho (\vn(x_0), \theta)$. For $p\in[2,\infty)$ we combine \Cref{prop:bound-Besov} (reiteration of Besov seminorms) with \eqref{eq:reg-min-p-big} to obtain
\begin{equation} \label{eq:reg-steps}
  \begin{split}
  |u|^p_{B^{s+\sigma/p}_{p,\infty} (D_\rho(x_0))} & \lesssim \sup_{h\in D} \frac{|u - u_h|^p_{W^s_p(D_\rho(x_0))}}{|h|^\sigma}
  \\
  & \lesssim \sup_{h\in D} \frac{\| u - T_h u\|^p_{\wt{W}^s_p(\Omega)}}{|h|^\sigma}
  \lesssim \omega(u; \calF, T, D, \sigma),
  \end{split}
\end{equation}
for $\sigma \in (0,1]$.
Instead for $p\in(1,2)$, combining \Cref{prop:bound-Besov} with \eqref{eq:reg-min-p-small} yields
\begin{equation} \label{eq:reg-steps-2}
      \begin{split}
  |u|^p_{B^{s+\sigma/2}_{p,\infty} (D_\rho(x_0))} & \lesssim \sup_{h\in D} \frac{\| u - T_h u\|^2_{\wt{W}^s_p(\Omega)}}{|h|^\sigma}
  \\
  & \lesssim \Big(\|u\|_{\wt{W}^s_p(\Omega)} + \|T_h u\|_{\wt{W}^s_p(\Omega)} \Big)^{2-p}\omega(u; \calF, T, D, \sigma).
      \end{split}
\end{equation}
This shows that proving that $\calF$ satisfies \Cref{def:regularity} ($(T,D,\sigma)$-regularity) for some $\sigma \in(0,1]$ gives rise to local regularity of solutions. 
    
\section{Regularity} \label{sec:regularity}
In this section we obtain regularity estimates for solutions to \eqref{eq:weak}. For that purpose, we analyze the regularity of the functional $\calF$ in \eqref{eq:def-energy} in the sense of Definition \ref{def:regularity}, and exploit the crucial property that the operator $T_h$ \eqref{eq:translation-operator} is locally a translation and $T_h v \in \widetilde{W}^s_p(\Omega)$ for all $v \in \widetilde{W}^s_p(\Omega)$. 

We split the energy in \eqref{eq:def-energy} as $\calF = \calF_G - \calF_1$, with 
\begin{equation} \label{eq:functionals} \begin{split}
& \calF_G(u) := \iint_{\Rd \times \Rd} G \left( x, y, \frac{u(x)-u(y)}{|x-y|^s} \right) \frac1{|x-y|^d} \, dy dx ,  \\
& \calF_1 (u) := \langle f, u \rangle
\end{split} \end{equation}
and recall that, by the subadditivity property \eqref{eq:subadditivity}, we can treat the two terms separately.

\subsection{Regularity of the functionals}\label{S:regularity-functionals}
As a first step towards deriving regularity estimates for minimizers of \eqref{eq:def-energy}, we prove the regularity of $\calF_1$ and $\calF_G$ with respect to the family of admissible outward vectors, cf. \Cref{def:outward_vectors}. We begin with an estimate on the linear part of the functional.

\begin{proposition}[regularity of $\calF_1$] \label{prop:regularity-f}
Let $q \in (1, \infty]$, $p \in (1,\infty)$, $p' = p/(p-1)$, $\sigma \in (0,1]$ and $t \in (-1,\sigma)$.
If $f \in B^{t}_{p',q'}(\Omega)$, then
$\calF_1$ is $(T, \calO_\rho(x_0),\sigma)$-regular in $\dot{B}^{\sigma-t}_{p,q}(\Omega)$ for all $x_0 \in \Omega$, $\rho > 0$.
Namely, for all $v \in \dot B^{\sigma-t}_{p,q}(\Omega)$
\begin{equation}\label{eq:regularity-f-Besov}
\sup_{h \in \calO_\rho(x_0) } \frac{\calF_1 (T_h v) - \calF_1 (v)}{|h|^{\sigma}} \lesssim \| f \|_{B^{t}_{p',q'}(D_{2\rho}(x_0) \cap \Omega)}  \| v \|_{B^{\sigma-t}_{p,q}(D_{2\rho}(x_0))}.
\end{equation}
\end{proposition}
\begin{proof}
We split the proof into three steps depending on the range of $\sigma$. Let $\sigma \in (0,1]$,
  $r \ge 0$, and $v \in \dot{B}^{r+\sigma}_{p,q}(\Omega)$.

\smallskip\noindent
1. {\it Case $t\in(-1,0]$}:
We use \eqref{eq:translation-operator} to write $T_h v = v \circ S_h$ with $S_h = I + h \phi$, and recall that $\mbox{supp}(\phi) \subset D_{2\rho}(x_0)$, so that 
$v - T_h v \in \dot{B}^{r}_{p,q}(D_{2\rho}(x_0))$. 
We thus obtain
\[
\calF_1 (T_h v) - \calF_1 (v) = \langle f, T_h v - v \rangle  \le \| f \|_{B^{-r}_{p',q'}(D_{2\rho}(x_0) \cap \Omega )}  \| T_h v - v \|_{\dot{B}^{r}_{p,q}(D_{2\rho}(x_0))}.
\]
Next, we resort to \eqref{eq:localized-translation-regularity-higher} to deduce 
\[
 \| T_h v - v \|_{B^{r}_{p,q}(B_{2\rho}(x_0))} \le |h|^{\sigma}  \| v \|_{B^{r+\sigma}_{p,q}(D_{2\rho}(x_0))},
\]
which implies for all $v \in \dot{B}^{r+\sigma}_{p,q}(\Omega)$
\begin{equation} \label{eq:reg-f-Besov-1}
\left| \calF_1 (T_h v) - \calF_1 (v) \right| \lesssim |h|^\sigma \| f \|_{B^{-r}_{p',q'}(D_{2\rho}(x_0) \cap \Omega )}  \|v\|_{B^{r+\sigma}_{p,q}(D_{2\rho}(x_0))} .
\end{equation}
This establishes \eqref{eq:regularity-f-Besov} upon setting $r = -t$.

\smallskip\noindent
2. {\it Case $t=\sigma$}: Let $f \in B^{\sigma}_{p',q'}(D_{2\rho}(x_0)\cap \Omega)$ and change variables to write
\[
\int_\Omega f(x) \, v(S_h(x)) \, dx = \int_{S_h(\Omega)} f(S_h^{-1}(x)) v(x) |\det \nabla S_h^{-1} (x)| \, dx,
\]
whence,
\[
\begin{aligned}
\left| \calF_1 (T_h v) - \calF_1 (v) \right|&= 
\left| \int_{B_{2\rho}(x_0) \cap \Omega} \left[ f(S_h^{-1}(x)) |\det \nabla S_h^{-1} (x)| - f(x) \right] v(x) \, dx \right| \\
& \lesssim \| v \|_{L^p(D_{2\rho}(x_0))} \| (1-|\det \nabla S_h^{-1}|) (f\circ S_h^{-1}) \|_{L^{p'}(D_{2\rho}(x_0) \cap \Omega)} \\
& + \| v \|_{L^p(D_{2\rho}(x_0))} \|f \circ S_h^{-1} - f \| _{L^{p'}(D_{2\rho}(x_0) \cap \Omega)},
\end{aligned}
\]
because $(1-|\det \nabla S_h^{-1}|) (f\circ S_h^{-1})$ and $f \circ S_h^{-1} - f$ vanish on $D_{2\rho}(x_0)^c$. 
We use Remark \ref{rmk:Sh} (properties of $S_h$ and $S_h^{-1}$) and \Cref{lemma:translation} (moduli of continuity) to deduce
\begin{equation} \label{eq:reg-f-Besov-2}
\left| \calF_1 (T_h v) - \calF_1 (v) \right| \lesssim |h|^{\sigma} \Vert f \Vert_{B^{\sigma}_{p',q'}(D_{2\rho}(x_0) \cap \Omega)} \Vert v \Vert_{L^p(D_{2\rho}(x_0))}.
\end{equation}

\smallskip\noindent
3. {\it Case $t\in(0,\sigma)$}: 
Since the mapping $(f,v) \to \calF_1 (T_h v) - \calF_1 (v)$ is bilinear, we may interpolate between \eqref{eq:reg-f-Besov-1} with $r = 0$ and \eqref{eq:reg-f-Besov-2}, with the same $p$ and $q$ in both expressions, to infer that
\eqref{eq:regularity-f-Besov} holds as well in this case.
\end{proof}

Next, we prove the regularity of the non-linear term $\calF_G$, defined in \eqref{eq:functionals}.

\begin{proposition}[regularity of $\calF_G$] \label{prop:regularity-FG}
Let $s \in (0,1)$ and assume that $G$ satisfies Hypothesis \ref{hyp:G} for some $p \in (1,\infty)$ and $\beta \in (0,1]$.
Then, the functional $\calF_G \colon \widetilde{W}^s_p(\Omega) \to \R$ defined in \eqref{eq:functionals} is $(T,\calO_\rho(x_0),\beta)$-regular in $\widetilde{W}^s_p(\Omega)$ for all $x_0 \in \Omega$, $\rho >0$.
Namely, for all $v \in \widetilde{W}^s_p(\Omega)$ it holds that
\begin{equation} \label{eq:regularity-FG}
\sup_{h \in \calO_\rho(x_0)} \frac{\calF_G (T_h v) - \calF_G (v)}{|h|^\beta} \lesssim \iint_{Q_{B_{2\rho}(x_0)}} \frac{|v(x) - v(y)|^p}{|x-y|^{d+ sp}} dydx,
\end{equation}
where $Q_{B_{2\rho}(x_0)} := (B_{2\rho}(x_0) \times \Rd) \cup (\Rd \times B_{2\rho}(x_0))$.
\end{proposition}
\begin{proof}
The change of variables $(x,y) \mapsto (S_h^{-1}(x), S_h^{-1}(y))=:(x_h,y_h)$ leads to
\[
\begin{aligned}
\calF_G(T_h v) - \calF_G(v) 
& = \iint_{\Rd \times \Rd}  \frac{G \left( x, y, \frac{T_h v(x)- T_h v(y)}{|x-y|^s} \right) - G \left( x, y, \frac{v(x)-v(y)}{|x-y|^s} \right)}{|x-y|^{d}} \, dy dx \\
&= \iint_{\Rd \times \Rd} \frac{G \left( x_h, y_h, \frac{v(x)- v(y)}{|x_h - y_h|^s} \right)}{|x_h - y_h|^{d}} \, J(x,y) dy dx \\
& \quad - \iint_{\Rd \times \Rd} G \left( x, y, \frac{v(x)-v(y)}{|x-y|^s} \right) \frac1{|x-y|^{d}} \, dy dx ,
\end{aligned}
\]
where $J(x,y) := |\det( \nabla S^{-1}_h(x)) \det(\nabla  S^{-1}_h(y))|$ for conciseness. We further split
\[
\calF_G(T_h v) - \calF_G(v) = I + II + III,
\]
with
\[ \begin{aligned}
I &:= \iint_{\Rd \times \Rd} \frac{G \left( x_h, y_h, \frac{v(x)- v(y)}{|x_h-y_h|^s} \right) - G \left( x_h, y_h, \frac{v(x)- v(y)}{|x - y|^s} \right)}{|x_h - y_h|^{d}} \, J(x,y) \, dy dx , \\
II &:= \iint_{\Rd \times \Rd} \frac{G \left( x_h, y_h, \frac{v(x)- v(y)}{|x - y|^s} \right)-G \left( x, y, \frac{v(x)- v(y)}{|x - y|^s} \right)}{|x_h - y_h|^{d}} \, J(x,y) \, dy dx, \\
III &:= \iint_{\Rd \times \Rd} G \left( x, y, \frac{v(x)- v(y)}{|x - y|^s} \right) \left(\frac{J(x,y)}{|x_h-y_h|^{d}} - \frac1{|x - y|^{d}} \right) dy dx.
\end{aligned} \]
We observe that, because $S_h$ is a one-to-one mapping on $\Rd$ that coincides with the identity over $B_{2\rho}(x_0)^c$, all the integrals above need to be computed on the set $Q_{B_{2\rho}(x_0)} = (B_{2\rho}(x_0) \times \Rd) \cup (\Rd \times B_{2\rho}(x_0))$.

Applying the Mean Value Theorem and using \eqref{eq:jacobian-Sh}, \eqref{eq:quasi-isometry-Sh} and the growth condition \eqref{eq:p-growth}, it follows that
\begin{equation} \label{eq:bound-I}
\left| I \right|\lesssim |h| \iint_{Q_{B_{2\rho}(x_0)}} \frac{|v(x) - v(y)|^{p}}{|x-y|^{d+sp}} dydx .
\end{equation}
Next, we resort to the $\beta$-H\"older continuity of $G$ with respect to the space variables \eqref{eq:G-Lipschitz}, use \eqref{eq:easy-Sh} and \eqref{eq:jacobian-Sh}, and obtain
\begin{equation} \label{eq:bound-II}
\begin{aligned}
 \left| II \right| & \lesssim  \iint_{Q_{B_{2\rho}(x_0)}} \left(|x_h - x|^{\beta} + |y_h-y|^\beta \right) \frac{|v(x) - v(y)|^p}{|x-y|^{d+sp}} dydx \\
&  \lesssim |h|^{\beta} \iint_{Q_{B_{2\rho}(x_0)}} \frac{|v(x) - v(y)|^p}{|x-y|^{d+ sp}} dydx .
\end{aligned}
\end{equation}
Finally, we use \eqref{eq:jacobian-Sh} and \eqref{eq:quasi-isometry-Sh} to derive
\[
\left|\frac{J(x,y)}{|S^{-1}_h(x) - S^{-1}_h(y)|^{d}} - \frac1{|x - y|^{d}} \right| \le \frac{|h|}{|x - y|^{d}}. 
\]
Therefore, combining this bound with condition \eqref{eq:p-growth}, we obtain
\begin{equation} \label{eq:bound-III}
\begin{aligned}
| III |& &  \lesssim |h| \iint_{Q_{B_{2\rho}(x_0)}} \frac{|v(x) - v(y)|^p}{|x-y|^{d+ sp}} dydx .
\end{aligned}
\end{equation}
Collecting \eqref{eq:bound-I}, \eqref{eq:bound-II}, and \eqref{eq:bound-III}, we deduce
\[
|\calF_G(T_h v) - \calF_G(v)| \lesssim |h|^\beta \iint_{Q_{B_{2\rho}(x_0)}} \frac{|v(x) - v(y)|^p}{|x-y|^{d+ sp}} dydx
\]
and conclude the proof.
\end{proof}

\begin{remark}[translation operator] \label{rmk:translation-operator}
The estimate \eqref{eq:regularity-FG} depends crucially on the structure of the translation operator \eqref{eq:translation-operator}. Indeed, localizing through a composition enabled us to perform a simple  argument based on change of variables. In contrast, if one aims to exploit the convexity of the functional $\calF_G$ by using a translation operator of the form
$\wt{T}_h = \phi v_h + (1-\phi) v $
for a suitable cutoff $\phi$, as in \cite{Savare98,BoNo21}, then one obtains a less accurate estimate where the right-hand side involves higher-order norms of $u$ \cite[Proposition 3.2]{BoNo21}.
\end{remark}

\subsection{Regularity of solutions}\label{S:regularity-solutions}
We are now in position to prove the regularity of minimizers of the energy \eqref{eq:def-energy} under Hypothesis \ref{hyp:G}. We first prove estimates in the Besov scale, and afterwards derive estimates in Sobolev norms by using embeddings. Finally, we show how to accommodate the theory to include finite-horizon operators.

\begin{theorem}[maximal Besov regularity] \label{thm:max-regularity}
Let $\Omega$ be a bounded Lipschitz domain, $s \in (0,1)$, $G \colon \R^d \times \R^d \times \R \to (0, \infty)$ satisfy Hypothesis \ref{hyp:G}, with $p \in (1,\infty)$, $\beta \in (0,1]$, and let $p' = \frac{p}{p-1}$.

If $p \ge 2$ and $f \in B^{-s+\frac{\beta}{p'}}_{p',1}(\Omega)$, then $u \in \dot B^{s+\frac{\beta}{p}}_{p,\infty}(\Omega)$ and
\begin{equation} \label{eq:max-reg-p-big}
\| u \|_{\dot B^{s+\frac{\beta}{p}}_{p,\infty}(\Omega)} \lesssim \| f \|_{B^{-s+\frac{\beta}{p'}}_{p',1}(\Omega)}^{\frac1{p-1}}.
\end{equation}

If $p < 2$ and $f \in B^{-s+\frac{\beta}{2}}_{p',1}(\Omega)$, then $u \in \dot B^{s + \frac{\beta}{2}}_{p,\infty}(\Omega)$ and
\begin{equation} \label{eq:max-reg-p-small}
\| u \|_{\dot B^{s + \frac{\beta}{2}}_{p,\infty}(\Omega)} \lesssim \|f\|_{W^{-s}_{p'}(\Omega)}^{\frac{2-p}{p-1}} \| f \|_{B^{-s+\frac{\beta}{2}}_{p',1}(\Omega)}.
\end{equation}
\end{theorem}
\begin{proof}
Let $\gamma = \frac{\beta}{p'}$ if $p \ge 2$ and $\gamma = \frac{\beta}2$ if $p < 2$. We first observe that problem \eqref{eq:Dirichlet} is well-posed, because $f \in B^{-s+\gamma}_{p',1}(\Omega) \subset W^{-s}_{p'}(\Omega)$ and, according to \eqref{eq:stability}, we have
\begin{equation} \label{eq:stability-Wsp}
\|u\|_{\widetilde{W}^s_p(\Omega)} \lesssim \|f\|_{W^{-s}_{p'}(\Omega)}^{\frac1{p-1}} \lesssim \|f\|_{B^{-s+\gamma}_{p',1}(\Omega)}^{\frac{p'}{p}}.
\end{equation}
We split the proof into three steps.

\smallskip\noindent
1. Because $\Omega$ is a Lipschitz domain, by Proposition \ref{prop:cone-property} (uniform cone property) there exist $\rho$, $\phii$ such that $\calC_\rho (\vn(x), \phii) \subset \calO_\rho(x)$ for all $x \in \Rd$. We consider a finite covering of $\Omega^\rho$ by balls $D_j = D(x_j, \rho)$ of radius $\rho$, $j = 1, \ldots , M$. By the localization estimate \eqref{eq:localization2}, it suffices to bound Besov seminorms over each of the balls $D_j$.

We consider one of the balls $D_j$ in the covering and set $\calC_j = \calC_\rho (\vn(x_j), \phii)$. Importantly, if $h \in \calC_j$ we can guarantee that $T_h u \in \widetilde{W}^s_p(\Omega)$.

Let $\sigma \in (0,\beta]$ and $t \in (-1,\sigma)$; our proof will use suitable bounds on the $(T,\calC_j,\sigma)$-regularity modulus of $\calF$. We exploit the subadditivity of this functional, and Propositions \ref{prop:regularity-f} (regularity of $\calF_1$) with $q=\infty$ and \ref{prop:regularity-FG} (regularity of $\calF_G)$): 
\begin{equation} \label{eq:regularity-continuity-modulus}
\begin{aligned}
& \omega(u; \calF, T, \calC_j, \sigma) \le \omega(u; \calF_G, T, \calC_j, \sigma) + \omega(u; \calF_1, T, \calC_j, \sigma) \\
& \lesssim  \iint_{Q_{B_{2\rho}(x_j)}} \frac{|u(x) - u(y)|^p}{|x-y|^{d+ sp}} dydx + \| f \|_{B^{t}_{p',1}(D_{2\rho}(x_j)\cap\Omega)}  \| u \|_{B^{\sigma-t}_{p,\infty}(D_{2\rho}(x_j))} .
\end{aligned} 
\end{equation}
In view of \eqref{eq:reg-steps} and \eqref{eq:reg-steps-2}, using \eqref{eq:stability-Wsp} we deduce
\begin{equation}\label{eq:first-bound}
   |u|_{B^{s+\sigma/q}_{p,\infty} (D_j)} \lesssim
    \|f\|_{W^{-s}_{p'}(\Omega)}^{\frac{q-p}{p-1}} \omega(u;\calF,T,\calC_j,\sigma)
\end{equation}
provided $q:=\max\{2,p\}$. We next distinguish between $p\in[2,\infty)$ and $p\in(1,2)$, and employ
  \eqref{eq:regularity-continuity-modulus} and \eqref{eq:first-bound} with $t=-s+\beta/q'$.

\smallskip\noindent
2. {\it Case $p \in [2,\infty)$}. Upon choosing $q=p$ in \eqref{eq:regularity-continuity-modulus}, we deduce
\[
| u |^p_{B^{s+\sigma/p}_{p,\infty}(D_j)} \lesssim  \iint_{Q_{B_{2\rho}(x_j)}} \frac{|u(x) - u(y)|^p}{|x-y|^{d+ sp}} dydx + \| f \|_{B^{t}_{p',1}(D_{2\rho}(x_j)\cap\Omega)}  \| u \|_{B^{\sigma-t}_{p,\infty}(D_{2\rho}(x_j))} .
\]
Adding over $j$ for $1\le J$, and using \Cref{lem:localization} (localization), specifically \eqref{eq:localization1} for the right-hand side and \eqref{eq:localization2} for the left-hand side, we obtain
\[
\| u \|^p_{\dot{B}^{s+\sigma/p}_{p,\infty}(\Omega)} \lesssim  \|u\|_{\widetilde{W}^{s}_{p}(\Omega)}^p  + \| f \|_{B^{t}_{p',1}(\Omega)} \| u \|_{\dot B^{\sigma-t}_{p,\infty}(\Omega)}
\]
for all $\sigma \in (0,\beta],\ t \in (-1,\sigma)$,
where the hidden constant depends on $J$. We now replace the first term on the right-hand side via the 
stability bound \eqref{eq:stability-Wsp}. This suggests the choices $t = -s + \frac{\beta}{p'} < \sigma\le\beta$
and yields the bound
\begin{equation} \label{eq:first-bound-p-big}
\| u \|^p_{\dot{B}^{s+\sigma/p}_{p,\infty}(\Omega)} \le C_1   \|f\|_{B^{-s+\beta/p'}_{p',1}(\Omega)}^{p'} + C_2 \| f \|_{B^{-s+\beta/p'}_{p',1}(\Omega)}  \| u \|_{\dot B^{\sigma+s-\beta/p'}_{p,\infty}(\Omega)}
\end{equation}
for suitable constants $C_1, C_2>0$ depending on $d,s,p,\Omega,J$. We observe that the differentiability index
$\sigma+s-\beta/p'$ of $u$ on the right-hand side gives a larger index $s+\sigma/p$ on the left-hand side
provided $\sigma\le\beta$. We thus view \eqref{eq:first-bound-p-big} as a recursion relation and set
$\sigma_0=0$, and for $k\ge0$ let $\sigma=\sigma_{k+1}\in(0,\beta]$ and
\[
s+\frac{\sigma_k}{p} = \sigma + s - \frac{\beta}{p'} \quad\Rightarrow\quad
\sigma_{k+1} := \frac{\sigma_k}p + \frac{\beta}{p'}.
\]
This implies $\sigma_{k} = \beta \left( 1 - \frac1{p^{k}} \right)$ and shows that
$\sigma_k\in(0,\beta)$ and $\sigma_k-t=\sigma_k+s-\beta/p'\ge s>0$ for all $k \ge 1$, as desired.
We claim that 
\begin{equation} \label{eq:iteration-claim-pge2}
\| u \|_{\dot{B}^{s+\sigma_k/p}_{p,\infty}(\Omega)} \le \Lambda_k \|f\|_{B^{-s+\frac{\beta}{p'}}_{p',1}(\Omega)}^{\frac{1}{p-1}}
\end{equation}
for some uniformly bounded constants $\Lambda_k$ for $k\ge0$. We argue by induction.
We first note that \eqref{eq:stability-Wsp} and the continuity of the embedding $\widetilde{W}^s_p(\Omega)\subset\dot{B}^{s}_{p,\infty}(\Omega)$ yield
\[
\| u \|_{\dot{B}^{s}_{p,\infty}(\Omega)} \le \Lambda_0 \|f\|_{B^{-s+\frac{\beta}{p'}}_{p',1}(\Omega)}^{\frac{1}{p-1}},
\]
for some $\Lambda_0 := \Lambda_0(d, s, p, \Omega)$; hence \eqref{eq:iteration-claim-pge2} is true for $k=0$.
We next set $\sigma = \sigma_{k+1}$ in \eqref{eq:first-bound-p-big} to arrive at
\[ \begin{aligned}
\| u \|^p_{\dot{B}^{s+\sigma_{k+1}/p}_{p,\infty}(\Omega)} & \le \left(  C_1  \|f\|_{B^{-s+\beta/p'}_{p',1}(\Omega)}^{\frac1{p-1}} + C_2 \| u \|_{\dot B^{s+\sigma_k/p}_{p,\infty}(\Omega)} \right) \| f \|_{B^{-s+\beta/p'}_{p',1}(\Omega)}  \\
& \le   \left(  C_1 + C_2 \Lambda_k \right)  \|f\|_{B^{-s+\beta/p'}_{p',1}(\Omega)}^{p'}.
\end{aligned}
\]
This shows that \eqref{eq:iteration-claim-pge2} holds for $\Lambda_{k+1} := (C_1 + C_2 \Lambda_k)^{1/p}$; it remains to prove that such a sequence is bounded. Let $\Lambda := \max \{ \Lambda_0, (p'C_1 + C_2^{p'})^{1/p} \}$. We obviously have $\Lambda_0 \le \Lambda$ and, if $\Lambda_k \le \Lambda$, applying Young's inequality we obtain
\[ \begin{split}
\Lambda_{k+1}^p & = C_1 + C_2 \Lambda_k \le C_1 + C_2 \Lambda \le C_1 + \frac{C_2^{p'}}{p'} + \frac{\Lambda^{p}}{p} \\
&  \le \frac1{p'}  \left( p' C_1 + C_2^{p'}  \right) + \frac{\Lambda^p}{p} \le \Lambda^p.
\end{split}\]
Thus, replacing $\Lambda_k$ by $\Lambda$ and letting  $k \to \infty$, we have $\sigma_k \to \beta$ and deduce the desired estimate \eqref{eq:max-reg-p-big} for $p\in[2,\infty)$.

\smallskip\noindent
3. {\it Case $p\in(1,2)$}. We choose $q=2$ in \eqref{eq:first-bound} and $t=-s+\beta/2$ in \eqref{eq:regularity-continuity-modulus}. After squaring, summing up over $j$ for $1\le j \le J$, and \Cref{lem:localization} (localization), we end up with
\[
\| u \|_{\dot{B}^{s+\sigma/2}_{p,\infty}(\Omega)}^2 \lesssim
\| f \|_{W^{-s}_{p'}(\Omega)}^{\frac{2-p}{p-1}} \left(|u|_{\widetilde{W}^{s}_{p}(\Omega)}^p + \| f \|_{B^{-s+\beta/2}_{p',1}(\Omega)}  \| u \|_{\dot B^{\sigma+s-\beta/2}_{p,\infty}(\Omega)} \right) 
\]
for all  $\sigma \in (0,\beta]$ such that $\sigma+s-\beta/2>0$; the hidden constant depends on $J$.
We next resort to \eqref{eq:stability-Wsp} to obtain the following counterpart of \eqref{eq:first-bound-p-big}:
\begin{equation} \label{eq:first-bound-p-small}
\| u \|_{\dot{B}^{s+\sigma/2}_{p,\infty}(\Omega)}^2 \le
\left( C_1 \| f \|_{W^{-s}_{p'}(\Omega)}^{p'} + C_2 \| f \|_{B^{-s+\beta/2}_{p',1}(\Omega)}  \| u \|_{\dot B^{\sigma+s-\beta/2}_{p,\infty}(\Omega)} \right)
\| f \|_{W^{-s}_{p'}(\Omega)}^{\frac{2-p}{p-1}}
\end{equation}
for all $\sigma \in (0,\beta]$ and with constants $C_1$ and $C_2$ that do not depend on $u$ or $f$. 

We now proceed as in Step 2 to exploit the improvement of the differentiability index from the right-hand side
to the left one in \eqref{eq:first-bound-p-small}. To this end, we set $\sigma_0=0, \sigma_{k+1} = \sigma\in (0,\beta]$ and rewrite
\[
\sigma + s-\frac{\beta}{2} = s + \frac{\sigma_k}{2}
\quad\Rightarrow\quad
\sigma_{k+1} = \frac{\beta+\sigma_k}{2} \quad\forall \, k\ge 0.
\]
This yields $\sigma_k = \beta(1 - 2^{-k})$, which satisfies the restrictions $\sigma_k \in (0,\beta)$,
$\sigma_k-t = \sigma_k + s-\frac{\beta}{2} \ge s>0$ for all $k\ge1$, and $\sigma_k\to\beta$.
We prove by induction that
\begin{equation} \label{eq:iteration-ple2-temp}
		\| u \|_{\dot{B}^{s+\frac{\sigma_k}2}_{p,\infty}(\Omega)} \le 
		\Lambda_{k} \| f \|_{W^{-s}_{p'}(\Omega)}^{\frac{2-p}{p-1} + 2^{-k}} 
		\| f \|_{B^{-s+\frac{\beta}{2}}_{p',1}(\Omega)}^{1 - 2^{-k}} \quad \forall k \ge 0, 
\end{equation}
with a uniformly bounded constant $\Lambda_k$. We first note that, according to \eqref{eq:stability-Wsp} and the continuity of the embedding $\widetilde{W}^s_p(\Omega)\subset\dot{B}^{s}_{p,\infty}(\Omega)$, we have
	\[
	\| u \|_{\dot{B}^{s}_{p,\infty}(\Omega)} \le \Lambda_0 \|f\|_{W^{-s}_{p'}(\Omega)}^{\frac{1}{p-1}}
	\]
	for some $\Lambda_0 := \Lambda_0(d, s, p, \Omega)$, whence \eqref{eq:iteration-ple2-temp} is true for $k=0$.
 	Suppose that \eqref{eq:iteration-ple2-temp} holds for some $k$, and set $\sigma = \sigma_{k+1}$ in \eqref{eq:first-bound-p-small} to obtain 
	\[ \begin{split}
		\| u \|_{\dot{B}^{s+\sigma_{k+1}/2}_{p,\infty}(\Omega)}^2 & \le
		\left( C_1 \| f \|_{W^{-s}_{p'}(\Omega)}^{p'} + C_2 \| f \|_{B^{-s+\beta/2}_{p',1}(\Omega)}  \| u \|_{\dot B^{s+\sigma_{k}/2}_{p,\infty}(\Omega)} \right) \| f \|_{W^{-s}_{p'}(\Omega)}^{\frac{2-p}{p-1}} \\
		& \le \left( C_1 \| f \|_{W^{-s}_{p'}(\Omega)}^{2 - 2^{-k}} + C_2 \Lambda_k \| f \|_{B^{-s+\beta/2}_{p',1}(\Omega)}^{2 - 2^{-k}}  \right)
		\| f \|_{W^{-s}_{p'}(\Omega)}^{\frac{2(2-p)}{p-1} + 2^{-k}} \\
		& \le \left( C_1 C_3^{2 - 2^{-k}} + C_2 \Lambda_k \right) 
		\| f \|_{B^{-s+\beta/2}_{p',1}(\Omega)}^{2 - 2^{-k}}  
		\| f \|_{W^{-s}_{p'}(\Omega)}^{\frac{2(2-p)}{p-1} + 2^{-k}},
	\end{split} \]
	where $C_3 := C_3(\Omega, d, s, p, \beta)$ is the constant of
	the continuous embedding $W^{-s}_{p'}(\Omega) \subset B^{-s+\beta/2}_{p',1}(\Omega)$.
	Since $1 \le 2 - 2^{-k} \le 2$, we have $C_3^{2 - 2^{-k}} \le \max\{C_3, C_3^2\} =: C_4$ which gives rise to \eqref{eq:iteration-ple2-temp} with constant
	\[
	\Lambda_{k+1} := \left(C_1 C_4 + C_2 \Lambda_k \right)^{1/2}.
	\]
	It only remains to show that $\Lambda_{k} \le \Lambda$ for some $\Lambda > 0$ and all $k \ge 0$. Let $\Lambda := \max\{ \Lambda_0, (2C_1 C_4 + C_2^2)^{1/2} \}$, then the same argument as in Step 2 is also valid in this setting.
	Finally, estimate \eqref{eq:max-reg-p-small} follows by letting $k \to \infty$.
\end{proof}

Theorem \ref{thm:max-regularity} gives the maximal Besov regularity one can expect via our variational
approach for solutions to problem \eqref{eq:Dirichlet}. We recall that \eqref{eq:max-reg-p-big} is optimal while
\eqref{eq:max-reg-p-small} is suboptimal for the 1$d$ profile $(1-x^2)_+^s$ in $\Omega=(-1,1)$. This is
due to the Hypothesis \ref{hyp:G} (monotonicity) for $p\in(1,2)$, which cannot be improved for a function
$\rho\mapsto G(\cdot,\cdot,\rho)$ convex and differentiable \cite[Remark 2.1]{Savare98}.

One can immediately deduce Sobolev regularity by combining \Cref{thm:max-regularity} with Lemma \ref{lemma:Besov-Sobolev-emb} (embedding). Even though such a lemma is stated for the spaces $B^{\sigma+\eps}_{p,\infty}(\Omega)$ and $W^{\sigma}_p(\Omega)$, the proof for the zero-extension spaces $\dot B^{\sigma+\eps}_{p,\infty}(\Omega)$ and $\wt W^{\sigma}_p(\Omega)$ follows by the same arguments.

\begin{corollary}[maximal Sobolev regularity] \label{C:Sobolev-delta}
  Let the assumptions of \Cref{thm:max-regularity} be satisfies. For $p\in[2,\infty)$ and
  $f \in B^{-s+\beta/p'}_{p',1}(\Omega)$ then
\begin{equation}\label{eq:sob-regularity-pLap-p>2}
  \| u \|_{\wt{W}^{s+\frac{\beta}{p} - \eps}_p(\Omega)} \lesssim \eps^{-\frac{1}{p}} \| f \|_{B^{-s+\frac{\beta}{p'}}_{p',1}(\Omega)}^{\frac1{p-1}}
\end{equation}
is valid provided $\eps \in (0, s + \beta/p)$. If $p\in(1,2)$ and $f \in B^{-s+\beta/2}_{p',1}(\Omega)$ then
\begin{equation} \label{eq:sob-regularity-pLap-p<2}
  \| u \|_{\wt{W}^{s+\frac{\beta}{2} - \eps}_p(\Omega)} \lesssim \eps^{-\frac{1}{p}} \|f\|_{W^{-s}_{p'}(\Omega)}^{\frac{2-p}{p-1}} \| f \|_{B^{-s+\frac{\beta}{2}}_{p',1}(\Omega)}
\end{equation}
holds provided $\eps \in (0, s + \beta/2)$.
\end{corollary}

\begin{remark}[interior Sobolev regularity]
Estimates \eqref{eq:sob-regularity-pLap-p>2} and \eqref{eq:sob-regularity-pLap-p<2} are valid up to the boundary of the domain. The obstruction to higher regularity is due to boundary behavior. In the superquadratic case $p\ge 2$, and for the $(p,s)$-fractional Laplacian \eqref{eq:def-pLap} --that gives rise to $\beta = 1$ in Hypothesis \ref{hyp:G}-- reference \cite{MR3558212} derives the following higher-order interior regularity provided $f \in W^s_{p'}(\Omega)$:
\[ \begin{split}
\mbox{if } s \le \frac{p-1}{p+1}, & \quad \mbox{then } u \in \bigcap_{\eps >0} W^{s \frac{p+1}{p-1} - \eps}_{p, loc}(\Omega), \\
\mbox{if } s > \frac{p-1}{p+1}, & \quad \mbox{then } u \in W^1_{p,loc}(\Omega) \mbox{ and } \nabla u \in \bigcap_{\eps >0} W^{s \frac{p+1}{p} - \frac{p-1}{p} - \eps}_{p, loc}(\Omega).
\end{split}\]
We point out that if $s < \frac12 - \frac1{2p}$, estimate \eqref{eq:sob-regularity-pLap-p>2} is actually stronger than the first statement above.
\end{remark}

The technique in Theorem \ref{thm:max-regularity} also allows one to derive regularity estimates in Besov spaces with a lower-order differentiability index whenever $f$ is less regular than in such a theorem. We shall not consider this procedure, but rather prove a lower regularity pickup by using interpolation theory. For that purpose, we need the following nonlinear interpolation estimate (cf. \cite[Th\'eor\`eme I.1]{Tartar}).

\begin{proposition}[nonlinear interpolation] \label{prop:nonlinear-interpolation}
Let $A_0 \subset A_1$, $B_0 \subset B_1$ be Banach spaces, $U \subset A_1$ a nonempty open set and $T \colon U \to B_1$ be a a function that maps $A_0 \cap U$ into $B_0$. Let us assume that there exist constants $c_0$, $c_1$ such that
\[ \begin{split}
& \| T f \|_{B_0} \le c_0 \| f \|_{A_0}^{\alpha_0} , \quad \forall f \in A_0 \cap U, \\
& \| T f - Tg \|_{B_1} \le c_1 \| f -g \|_{A_1}^{\alpha_1} , \quad \forall f,g \in U
\end{split} \]
for some $\alpha_0 > 0$, $\alpha_1 \in (0,1]$. Then, if $\theta \in (0,1)$ and $q \in [1,\infty]$, $T$ maps $(A_0, A_1)_{\theta,q}$ into $(B_0, B_1)_{\eta, r}$, where $\frac{1-\eta}{\eta} = \frac{\alpha_1}{\alpha_0} \frac{1-\theta}{\theta}$ and $r = \max\{1, \frac{q}{(1-\eta)\alpha_0+\eta\alpha_1} \}$.
\end{proposition}

\begin{corollary}[regularity pickup for rough data] \label{cor:less-regular-data}
  Let $\Omega$ be a bounded Lipschitz domain, $G \colon \R^d \times \R^d \times \R \to (0, \infty)$ satisfy Hypothesis \ref{hyp:G}, with $p \in (1,\infty)$, $\beta \in (0,1]$, and let $p' = \frac{p}{p-1}$, and $\theta \in (0,1)$. Then, the solution operator $f \mapsto u$ is bounded between the following spaces
\[
\begin{aligned}
\mbox{ if } p \ge 2 \mbox{ and } f \in W^{-s+\theta \frac\beta{p'}}_{p'}(\Omega) &\ \Rightarrow \ u \in \widetilde{W}^{s+ \theta \frac\beta{p}}_p(\Omega); \\
\mbox{ if } 1 < p < 2 \mbox{ and } f \in W^{-s+\theta \frac\beta2}_{p'}(\Omega) & \ \Rightarrow \ u \in \widetilde{W}^{s+ \theta \frac\beta2}_{p'}(\Omega).
\end{aligned}
\]
\end{corollary}
\begin{proof}
The proof follows by a direct application of Proposition \ref{prop:nonlinear-interpolation}. For $p \ge 2$, we combine \eqref{eq:max-reg-p-big} and \eqref{eq:Holder-solution-op}, while for $1<p<2$ we combine \eqref{eq:max-reg-p-small} and \eqref{eq:local-Lip-solution-op}. 
\end{proof}

\subsection{Operators with finite horizon}
%
Thus far, we have obtained regularity estimates under Hypothesis \ref{hyp:G}. In particular, that hypothesis dictates the smoothness, growth and behavior at $x=y$ of $G(x,y,\rho)$. However, such global behavior constraints can be significantly relaxed to incorporate, for instance, {\em finite-horizon operators.} For simplicity, we now assume $G$ takes the form
\begin{equation} \label{eq:def-G-finite-horizon}
G(x,y,\rho) = \frac{1}{2p} \, \psi\left(\frac{|x-y|}\delta \right) \, |\rho|^p,
\end{equation}
where $1 < p < \infty$ and $\psi : [0,\infty) \to [0,\infty)$ is a given function.
In case $\psi$ is supported in the unit interval $[0,1]$, the parameter $\delta >0$ above is the horizon of the resulting operator
\begin{equation} \label{eq:FH-operator}
\calA_\delta u (x) := \int_{\Rd} \psi\left(\frac{|x-y|}\delta \right) \frac{|u(x)-u(y)|^{p-2} (u(x)-u(y))}{|x-y|^{d+sp}} \, dy .
\end{equation}

Let us assume that $\psi \in L^\infty(\Rd)$.
Then, the choice \eqref{eq:def-G-finite-horizon} trivially fulfills conditions {\bf convexity}, {\bf$p$-growth},  and {\bf continuity}; moreover, we note it satisfies the {\bf symmetry} condition. The only two missing assumptions from Hypothesis \ref{hyp:G} in this setting are {\bf space continuity} and {\bf monotonicity}. At this point, we can define the {\em energy norm} induced by $G$,
\[
\iii{v} := \langle \calA_\delta v, v \rangle^\frac1p = \left(\frac{1}{2} \iint_{\Rd\times\Rd} \psi\left( \frac{|x-y|}{\delta} \right) \frac{|v(x)-v(y)|^p}{|x-y|^{d+sp}} \, d y dx \right)^{\frac1p}
\]
and realize it satisfies
\begin{equation} \label{eq:FH-upper-bound}
\iii{v} \le \left(\frac{\| \psi \|_{L^\infty(\Rd)}}{C_{d,s,p}}\right)^\frac1p  \, \|v\|_{\widetilde{W}^s_p(\Omega)} \quad \forall v \in \widetilde{W}^s_p(\Omega),
\end{equation}
where $C_{d,s,p}$  is the constant given in \eqref{eq:Cdsp}.

Some form of non-degeneracy is needed in order to have a reverse inequality to \eqref{eq:FH-upper-bound}. We assume there exists some $r > 0$ such that $\psi \ge \psi_0 > 0$ on the interval $[0,r]$. This implies the following property:

\smallskip

\noindent {\bf 4'. Local monotonicity.} There exists $\alpha > 0$ such that for all $x,y \in \Rd$ with $|x-y|\le r\delta$ and all $\rho, \rho' \in \R$,
\[ \begin{aligned}
& (G_\rho (x,y,\rho) - G_\rho (x,y,\rho')) (\rho - \rho') \ge \alpha |\rho - \rho'|^p & \quad \mbox{ if } p \ge 2, \\ 
&  (G_\rho (x,y,\rho) - G_\rho (x,y,\rho')) (\rho - \rho') \ge \alpha |\rho - \rho'|^2 | |\rho| + |\rho'| |^{p-2} & \quad \mbox{ if } 1 < p< 2.
\end{aligned}\]

\begin{lemma}[non-degeneracy]\label{L:non-degeneracy}
Let $\psi \ge \psi_0 > 0$ on the interval $[0,r]$. Then, there exists a constant $C=C(d,p,s,\Omega,r, \psi_0, \delta)$ such that
\begin{equation} \label{eq:FH-lower-bound}
\|v\|_{\widetilde W^s_p(\Omega)} \le C \, \iii{v}  \quad \forall v \in \widetilde W^s_p(\Omega). 
\end{equation}
\end{lemma}
\begin{proof}
We invoke the localization estimate of \cite[Lemma 7]{DyKa13}, which reads
\begin{align*} \label{eq:FH-chain-bound}
  \iint_{B_R(0) \times B_R(0)} & \frac{|v(x)-v(y)|^p}{|x-y|^{d+sp}} \, d y dx
  \\
  & \le \left(\frac{3R}{r \delta} \right)^{p(1-s)} \iint_{B_R(0) \times B_R(0)} \frac{|v(x)-v(y)|^p}{|x-y|^{d+sp}} \, \chi_{ \{|x-y| \le r \delta\} } \, d y dx,
\end{align*}
and is valid for all $R > r \delta > 0$ and $s \in (0,1)$. This implies
\begin{equation*}
  \iint_{B_R(0) \times B_R(0)} \frac{|v(x)-v(y)|^p}{|x-y|^{d+sp}} \, d y dx
  \le \frac{2(3R)^{p(1-s)}}{(r \delta)^{p(1-s)} \psi_0 }\iii{v}^p .
\end{equation*}
regardless of whether or not the support of $\psi$ is compact.
Given $v \in \widetilde W^s_p(\Omega)$, we use this bound with $R>0$ sufficiently large so that $\Omega \subset B_R(0)$ and $\mbox{dist}(\Omega, \partial B_R(0)) \ge \frac{R}2$, exploit the fact that $|x-y| \ge  \frac{R}2$ for all $x \in \Omega$ and $y \in B_R(0)^c$ and integrate in polar coordinates to get
\begin{align*}
\|v\|_{\widetilde W^s_p(\Omega)}^p & = \frac{C_{d,s,p}}{2}\iint_{B_R(0) \times B_R(0)} \frac{|v(x)-v(y)|^p}{|x-y|^{d+sp}} \, d y dx
\\
& + C_{d,s,p} \iint_{B_R(0) \times B_R(0)^c} \frac{|v(x)|^p}{|x-y|^{d+sp}} \, d y dx
  \\
& \le \frac{(3R)^{p(1-s)} C_{d,s,p}}{(r \delta)^{p(1-s)} \psi_0 } \iii{v}^p +  \frac{2^{sp} C_{d,s,p} \omega_{d-1} }{s p \, R^{sp}} \| v \|_{L^p(\Omega)}^p,
\end{align*}
where $\omega_{d-1}=|\mathbb{S}^{d-1}|$ denotes the $(d-1)$-dimensional measure of the unit sphere $\mathbb{S}^{d-1}=\partial B_1(0)$ in $\Rd$. We next use the well-known Poincar\'e inequality
\[
\| v \|_{L^p(\Omega)} \le C(\Omega, p) \|v\|_{\widetilde W^s_p(\Omega)} \qquad \forall v \in \widetilde{W}^s_p(\Omega), 
\]
and take $R>0$ sufficiently large such that
\[
\frac{2^{sp} C_{d,s,p} \omega_{d-1} }{s p \, R^{sp}} \| v \|_{L^p(\Omega)}^p \le \frac12 \|v\|_{\widetilde W^s_p(\Omega)}^p \quad \forall v \in \widetilde W^s_p(\Omega),
\]
to  obtain a constant $C=C(d,p,s,\Omega,r, \psi_0, \delta)$ such that \eqref{eq:FH-lower-bound} holds. 
\end{proof}

By combining \eqref{eq:FH-upper-bound} and \eqref{eq:FH-lower-bound}, we deduce that the energy norm $\iii{\cdot}$ is equivalent to the $\widetilde W^s_p(\Omega)$-norm. Consequently, the Dirichlet problem for the operator $\mathcal A_{\delta}$ defined in \eqref{eq:FH-operator} is well-posed in $\widetilde W^s_p(\Omega)$ uniformly in $s$: if $f \in W^{-s}_{p'}(\Omega)$, there exists a unique $u \in \widetilde W^s_p(\Omega)$ satisfying
\begin{equation*}\label{eq:weak-delta}
\langle \mathcal A_{\delta} u , v \rangle = \langle f , v \rangle \quad \forall v \in \widetilde W^s_p(\Omega),
\end{equation*}
and we have the stability bound $\| u \|_{\widetilde W^s_p(\Omega)} \lesssim \|f\|_{W^{-s}_{p'}(\Omega)} $.

Another consequence of the equivalence between the energy norm $\iii{\cdot}$ and the $\widetilde W^s_p(\Omega)$-norm is that the variational approach of Sections \ref{S:regularity-functionals} and \eqref{S:regularity-solutions} hinges on the regularity of $\calF_1$ and $\calF_G$ and still applies in this context regardless of the support of $\psi$. We state this next.

\begin{corollary}[operator with H\"older continuous $\psi$]\label{C:FH-smooth-phi}
Let $\psi$ in \eqref{eq:def-G-finite-horizon} be globally $\beta$-H\"older continuous with $\beta\in(0,1]$ and satisfy $\psi \ge \psi_0 > 0$ on the interval $[0,r]$ for $r>0$. Then the maximal regularity estimates \eqref{eq:max-reg-p-big} for $2\le p <\infty$ and \eqref{eq:max-reg-p-small} for $1<p<2$ are valid regardless of the support of $\psi$.
\end{corollary}

\begin{remark}[tempered fractional Laplacians]
Besides being applicable to finite-horizon operators with $\beta$-H\"older continuous kernel, the previous result is valid for a family of  {\em tempered} fractional $p$-Laplacians. Concretely, we let  $\lambda = \frac1\delta>0$ and $\psi(\rho) = e^{-\lambda \rho}$ to obtain
\[
\calA_\lambda u (x):= \int_\Rd  \frac{|u(x) - u(y)|^{p-2} (u(x) - u(y))}{e^{\lambda|x-y|}\, |x-y|^{d+sp}} \, dy .
\]
In the linear setting ($p=2$), this operator arises from the study of tempered L\'evy flights and has been investigated for example in \cite{Deng:18}. For the homogeneous Dirichlet problem associated to the operator $\calA_\lambda$ above, the regularity estimates in Theorem \ref{thm:max-regularity} are valid with $\beta = 1$.
\end{remark}

Finite-horizon operators in practice usually involve a discontinuous function $\psi$ such as the characteristic
function of $[0,r]$. This does not fit within the preceding theory but maximal Besov regularity is still valid,
at least for linear operators, provided $\psi$ is H\"older in a neighborhood of the origin. We explore this next,
but before we point out that even local regularity seems excessive, an interesting question to investigate.

We need to make the following local regularity assumption to compensate for the lack of {\bf space continuity}
hypothesis of $G$:

\smallskip
\noindent {\bf 2'. Local regularity.} There exists $r > 0$ such that $\psi$ is of class $C^\beta$ on the interval $[0,r]$
  for some $\beta \in (0,1]$, and $\psi \ge \psi_0 > 0$ on $[0,r]$.

\begin{theorem}[linear finite-horizon operator with discontinuous $\psi$]\label{T:FH-smooth-phi}
  Let $\psi$ satisfy the previous local regularity assumption with some $\beta \in (0,1].$
If $G(x,y,\rho)$ is of the form \eqref{eq:def-G-finite-horizon} with $p=2$, then the
following maximal regularity holds
\begin{equation}\label{eq:lift-finite-horizon}
  \|u\|_{\dot{B}^{s+\frac\beta2}_{2,\infty}(\Omega)} \lesssim \|f\|_{B^{-s+\frac\beta2}_{2,1}(\Omega)}.
\end{equation}
\end{theorem}
\begin{proof}
We resort to a perturbation argument. We proceed in several steps.

\medskip\noindent
1. {\it Perturbation}:
  Let $\wt{\psi}\in C^\beta[0,\infty)$ coincide with $\psi$ on $[0,r]$
and rewrite the equation $\calA_\delta u = f$ with $u=0$ in $\Omega^c$ as
\[
\wt{\calA}_\delta u = f + \big( \wt{\calA}_\delta u - \calA_\delta u \big) = \wt{f}.
\]
The operator $\calB_\delta := \wt{\calA}_\delta - \calA_\delta$ is a convolution operator that reads as follows in terms of the function $\vp := \wt\psi - \psi$, which vanishes on $[0,r]$:
\[
\calB_\delta u(x) = \int_{\R^d} \vp \Big( \frac{|z|}{\delta} \Big) \frac{u(x) - u(x-z)}{|z|^{d+2s}} dz
  = K u(x) + k * u(x).
\]
Above, $k(z) = \vp \big(\frac{|z|}{\delta} \big) |z|^{-d-2s}$ and $K = \int_{\R^d} k(z) dz<\infty$ because
$\psi \in L^\infty(\R^d)$.

\medskip\noindent
2. {\it Properties of $\calB_\delta$}: $\calB_\delta:\dot{B}^t_{2,\infty}(\Omega) \to {B}^t_{2,\infty}(\Rd)$
is a linear bounded operator
\[
\|\calB_\delta u \|_{\dot{B}^t_{2,\infty}(\Omega)} \lesssim \| u \|_{\dot{B}^t_{2,\infty}(\Omega)}
\quad \forall \, 0 < t < 1.
\]
It suffices to examine $k * u$ which, using its linear structure and Young's inequality with
$K=\|k\|_{L^1(\R^d)}$, yields
\[
\|k*u(\cdot+h)-k*u\|_{L^2(\R^d)} \le K \|u(\cdot+h)-u\|_{L^2(\R^d)} \quad\forall\, h\in\R^d.
\]
Consequently, we deduce the asserted estimate from
\[
\frac{\|\calB_\delta u(\cdot+h) - \calB u\|_{L^2(\R^d)}}{|h|^t} \le K
\frac{\|u(\cdot+h)-u\|_{L^2(\R^d)}}{|h|^t} \quad \forall\, 0 < t  < 1.
\]

\medskip\noindent
3. {\it Regularity of functionals}: Combining the local estimate \eqref{eq:regularity-f-Besov} of
Proposition \ref{prop:regularity-f} (regularity of $\calF_1$)  for $\wt{f}$, the argument of 
Theorem \ref{thm:max-regularity} (maximal regularity) leading to \eqref{eq:regularity-continuity-modulus} and
\eqref{eq:first-bound-p-big} yields the following estimate for any $\sigma\in(0, \beta]$ and $t\in(-s,\sigma)$
\begin{equation*}
\|u\|_{\dot{B}^{s+\sigma/2}_{2,\infty}(\Omega)}^2 \lesssim 
 \|f\|_{B^t_{2,1}(\Omega)}^2 + \|f\|_{B^t_{2,1}(\Omega)} \|u\|_{B^{\sigma-t}_{2,\infty}(\Omega)}
+ \|\calB_\delta u\|_{B^t_{2,1}(\Omega)} \|u\|_{B^{\sigma-t}_{2,\infty}(\Omega)},
\end{equation*}
where we have used the definition of $\wt{f}$ and Step 2 to obtain the last term. To be able to iterate
this estimate we observe that $\|\calB_\delta u\|_{B^t_{2,1}(\Omega)} \lesssim \|\calB_\delta u\|_{B^{t+2\epsilon}_{2,\infty}(\Omega)}$ for any
$\epsilon>0$ and to simplify the subsequent derivation we take $\epsilon = \frac{s}{2}$ so that $t+2\epsilon = t + s > 0$. We thus have $\|\calB_\delta u\|_{B^t_{2,1}(\Omega)} \lesssim \| u \|_{B^{t+2\epsilon}_{2,\infty}(\Omega)}$
and 
\begin{equation}\label{eq:basic-recurrence}
\|u\|_{\dot{B}^{s+\sigma/2}_{2,\infty}(\Omega)}^2 \lesssim
\|f\|_{B^t_{2,1}(\Omega)}^2 + \|f\|_{B^t_{2,1}(\Omega)} \|u\|_{B^{\sigma-t}_{2,\infty}(\Omega)}
+ \|u\|_{B^{t+s}_{2,\infty}(\Omega)} \|u\|_{B^{\sigma-t}_{2,\infty}(\Omega)}.
\end{equation}

\medskip\noindent
4. {\it Preliminary regularity}: We now claim that $u \in \dot{B}^{\frac\beta2}_{2,\infty}(\Omega)$ with 
\begin{equation} \label{eq:partial-regularity}
\| u \|_{\dot{B}^{\frac\beta2}_{2,\infty}(\Omega)} \lesssim  \|f\|_{B^{-s+\frac\beta2}_{2,1}(\Omega)}.
\end{equation}
If $s \ge \frac\beta2$ this is straightforward, because of the stability bound $\| u \|_{\wt{H}^s(\Omega)} \lesssim \|f\|_{H^{-s}(\Omega)}$ and the continuity of the embedding $\wt{H}^s(\Omega) \subset \dot{B}^{\beta/2}_{2,\infty}(\Omega)$.

In the case $s < \frac\beta2$, we iterate \eqref{eq:basic-recurrence}. We set $\sigma_0:=0$ and
\begin{equation} \label{eq:def-sigma}
\sigma_{k+1} := \sigma_k + \frac{s}{2}.
\end{equation}
In the first few iterations we could have $0<\sigma_k \le -s + \frac\beta2$, in which case we set $\sigma := 2 \sigma_{k+1} = s +2\sigma_k \in [s,\beta-s]\subset (0, \beta]$ and $t := \sigma_k = \frac{\sigma -s}2 \in [0,\frac{\sigma}{2})\subset(-s,\sigma)$. This choice of parameters yields $\sigma - t = t + s = s + \sigma_k$, whence \eqref{eq:basic-recurrence} reads
\[
\|u\|_{\dot{B}^{s+\sigma_{k+1}}_{2,\infty}(\Omega)}^2 \lesssim
\|f\|_{B^t_{2,1}(\Omega)}^2 + \|f\|_{B^t_{2,1}(\Omega)} \|u\|_{B^{s + \sigma_k}_{2,\infty}(\Omega)}
+ \|u\|_{B^{s + \sigma_k}_{2,\infty}(\Omega)}^2.
\]
Additionally, continuity of the embedding $B^{-s+\frac\beta2}_{2,1}(\Omega) \subset B^t_{2,1}(\Omega)$, gives
\[
\|u\|_{\dot{B}^{s+\sigma_{k+1}}_{2,\infty}(\Omega)}^2 \lesssim
\|f\|_{B^{-s+\frac\beta2}_{2,1}(\Omega)}^2 + \|u\|_{B^{s + \sigma_k}_{2,\infty}(\Omega)}^2,
\]
and the bound $\|u\|_{\dot{B}^{s+\sigma_k}_{2,\infty}(\Omega)} \lesssim\|f\|_{B^{-s+\frac\beta2}_{2,1}(\Omega)}$
valid for $k=0$ implies
\[
\|u\|_{\dot{B}^{s+\sigma_{k+1}}_{2,\infty}(\Omega)} \lesssim
\|f\|_{B^{-s+\frac\beta2}_{2,1}(\Omega)} ,
\]
for as long as $\sigma_k \le -s +\frac\beta2$.
Moreover, because $\sigma_{k+1} = \sigma_k + \frac{s}{2} ,$ we have a regularity improvement by the fixed amount $\frac{s}{2}$ in each iteration. Therefore, after a finite (but $s$-dependent) number $k_*$ of iterations, we reach $\sigma_{k_*} > -s +\frac\beta2$ and deduce the validity of the regularity bound \eqref{eq:partial-regularity}.

\medskip\noindent
5. {\it Final regularity}: We now assume \eqref{eq:partial-regularity} and define the new sequence
\[
\sigma_0 := 0, \quad \sigma_{k+1} := \frac\beta4 + \frac{\sigma_k}2 \quad \Rightarrow \quad \sigma_{k} = \frac\beta2 \left( 1 - \frac{1}{2^{k}}\right) \to \frac\beta2 .
\]
We fix $\sigma = \sigma_k + \frac\beta2 \in (0, \beta]$ and $t = - s + \frac\beta2 \in (-s, \sigma)$ in \eqref{eq:basic-recurrence}, and note that $s + \sigma/2 = s +\sigma_{k+1},$ $\sigma - t = s + \sigma_k,$ and $t + s = \frac\beta2$, to arrive at
\[
\|u\|_{\dot{B}^{s+\sigma_{k+1}}_{2,\infty}(\Omega)}^2 \lesssim
\|f\|_{B^{-s+\frac\beta2}_{2,1}(\Omega)}^2 + \|f\|_{B^{-s+\frac\beta2}_{2,1}(\Omega)} \|u\|_{B^{s+\sigma_k}_{2,\infty}(\Omega)}
+ \|u\|_{B^{\frac\beta2}_{2,\infty}(\Omega)}  \|u\|_{B^{s+\sigma_k}_{2,\infty}(\Omega)}.
\]
Using \eqref{eq:partial-regularity}, we get
\begin{equation*}\label{eq:second-recurrence}
\|u\|_{\dot{B}^{s+\sigma_{k+1}}_{2,\infty}(\Omega)}^2 \le
\left( C_1 \|f\|_{B^{-s+\frac\beta2}_{2,1}(\Omega)} + C_2 \|u\|_{B^{s+\sigma_{k}}_{2,\infty}(\Omega)} \right) \|f\|_{B^{-s+\frac\beta2}_{2,1}(\Omega)},
\end{equation*}
with constants $C_1,C_2$ depending on $\Omega, d$ and $s$. We finally proceed as in the proof of Theorem \ref{thm:max-regularity} (maximal Besov regularity) to show
\[
\|u\|_{\dot{B}^{s+\sigma_{k+1}}_{2,\infty}(\Omega)} \le 
\Lambda \|f\|_{B^{-s+\frac\beta2}_{2,1}(\Omega)}
\]
for a finite number $\Lambda > 0$. We trivially have 
$
\|u\|_{\dot{B}^{s+\sigma_{0}}_{2,\infty}(\Omega)} \le \Lambda_0 \|f\|_{B^{-s+\beta/2}_{2,1}(\Omega)}
$
for some $\Lambda_0 (d,s,\Omega) > 0$. We define $\Lambda_{k+1} := (C_1 + C_2 \Lambda_k)^{1/2}$ and realize that
\[
\|u\|_{\dot{B}^{s+\sigma_{k+1}}_{2,\infty}(\Omega)} \le \Lambda_{k+1} \|f\|_{B^{-s+\frac\beta2}_{2,1}(\Omega)}.
\]
Since $\Lambda_{k+1} \le \max\{ \Lambda_0, (2C_1+C_2)^{1/2} \} =: \Lambda$, passing to the limit $k \to \infty$, the desired estimate \eqref{eq:lift-finite-horizon} follows immediately.
\end{proof}

\begin{remark}[truncated fractional Laplacians]\label{R:Sobolev-delta}
Estimate \eqref{eq:lift-finite-horizon}  holds whenever $\psi$ is locally $\beta$-H\"older continuous at the origin. Consequently, it applies with $\beta = 1$ to (linear) {\em truncated} fractional Laplacians \cite{Burkovska:19}
\[
\calA u (x) := C(d,s,\delta) \int_{B_\delta (x)} \frac{u(x) - u(y)}{|x-y|^{d+2s}} \, dy ,
\]
for which $\psi (\rho) = \chi_{[0,1]}(\rho)$.
\end{remark}

In the same fashion as \Cref{C:Sobolev-delta}, the following maximal Sobolev regularity holds for operators of the form \eqref{eq:FH-operator} and, in particular, for linear truncated fractional Laplacians and tempered fractional $p$-Laplacians. 
\begin{corollary}[maximal Sobolev regularity]\label{C:FH}
Under the hypothesis of either \Cref{C:FH-smooth-phi} for any $p\in(1,\infty)$
    or of \Cref{T:FH-smooth-phi} for $p=2$,
for all $\eps>0$ sufficiently small and $q = \max\{ 2, p \}$ there holds
\begin{equation*}\label{eq:FH-BesovSobolev}
  \| u \|_{\wt{W}^{s+\frac{\beta}{q} - \eps}_p(\Omega)} \lesssim \eps^{-\frac{1}{p}} 
   \|f\|_{W^{-s}_{p'}(\Omega)}^{\frac{q-p}{p-1}}  \| f \|_{B^{-s+\frac{\beta}{q'}}_{p',1}(\Omega)}^{\frac1{q-1}}.
\end{equation*}
\end{corollary}

\section{Approximation} \label{sec:FE}
As an application of our regularity estimates, we consider discretizations of the problem
\eqref{eq:Dirichlet} by means of the finite element method with piecewise linear continuous functions. From now on, we assume that $G$ satisfies Hypothesis \ref{hyp:G}.

Let $h_0 > 0$; for $h \in (0, h_0]$, we let $\mathcal{T}_h$ denote a triangulation of $\Omega$, i.e., $\mathcal{T}_h = \{T\}$ is a partition of $\Omega$ into simplices $T$ of diameter $h_T$ and $h = \max_{T \in \Th} h_T$. 
We assume the family $\{\Th \}_{h>0}$ to be shape-regular, namely,
\[
\sigma := \sup_{h>0} \max_{T \in \Th} \frac{h_T}{\rho_T} <\infty,
\]
where $\rho_T $ is the diameter of the largest ball contained in $T$. We take elements to be closed sets.

Let $\mathcal{N}_h$ be the set of interior vertices of $\Th$, $N$ be its cardinality and $\{ \varphi_i \}_{i=1}^N$ be the standard piecewise linear Lagrangian basis, with $\phii_i$ associated to the node $\x_i \in \mathcal{N}_h$. With this notation, the set of discrete functions is
\begin{equation*} \label{eq:FE_space}
\wt{\mathbb{V}}_h :=  \left\{v: \Rd \to \R \colon v \in C_0(\Rd), \ v = \sum_{i=1}^N v_i \varphi_i \right\},
 \end{equation*}
where $v$ is trivially extended by zero outside $\Omega$. It is clear that $\wt{\mathbb{V}}_h \subset \widetilde{W}^s_p(\Omega)$ for all $s \in (0,1), p \in (1,\infty)$. Therefore, we consider a direct finite element discretization and seek $u_h \in \wt{\mathbb{V}}_h$ such that
\begin{equation} \label{eq:weak-discrete}
\iint_{\Rd\times\Rd} \widetilde G \left( x, y, \frac{u_h(x)-u_h(y)}{|x-y|^s} \right) \frac{(u_h(x)-u_h(y))(v_h(x)-v_h(y))}{|x-y|^{d+2s}} \, dx \, dy = \langle f , v_h \rangle
\end{equation}
for all $v_h \in \wt{\mathbb{V}}_h$, where we recall that $\widetilde G (x,y, \rho) = G_\rho (x,y,\rho)/\rho$. Clearly, $u_h$ solves \eqref{eq:weak-discrete} if and only if it is the minimizer of the restriction of the convex functional $\calF$ from \eqref{eq:def-energy} over the linear space $\wt{\mathbb{V}}_h$; existence of discrete solutions follows immediately. Moreover, if we take $v_h = u_h$ in \eqref{eq:weak-discrete}, then we immediately obtain the discrete stability bound
\begin{equation} \label{eq:discrete_stability}
\|u_h\|_{\widetilde W^s_p(\Omega)} \lesssim \| f \|_{W^{-s}_{p'}(\Omega)}^{\frac1{p-1}}.
\end{equation}

\subsection{Localization and interpolation estimates}

The seminorm $|\cdot|_{W^s_p(\Rd)}$ is nonlocal and, consequently, is non-additive with respect to domain partitions. To localize it, we first define the star (or patch) of a set  $A \in \Omega$ by
\[
  S_A^1 := \bigcup \big\{ T \in \Th \colon T \cap A \neq \emptyset \big\}.
\]
Given $T \in \Th$, the star $S_T^1$ of $T$ is the first ring of $T$ and the star $S_T^2$ of $S_T^1$
is the second ring of $T$. We have the following localization estimate, that can be proved with the same arguments as in \cite{Faermann2, Faermann}:
\begin{equation} \label{eq:localization-F}
  |v|_{W^s_p(\Omega)}^p \leq \sum_{T \in \Th} \left[ \int_T \int_{S_T^1} \frac{|v (x) - v (y)|^p}{|x-y |^{d+sp}} \; dy \; dx + C(d,\sigma) \frac{2^p}{s p h_T^{sp}} \| v \|^p_{L^p(T)} \right]
\end{equation}
for all $ v \in W^s_p(\Omega)$.

This localization of fractional-order seminorms implies that, in order to prove global interpolation estimates in $W^s_p(\Omega)$, it suffices to produce over the set of patches $\{T \times S_T^1 \}_{T \in \Th}$ plus local, zero-order contributions.

We point out, however, that clearly if one wants to have a zero-extension norm on the left hand side in \eqref{eq:localization-F}, then interactions between $\Omega$ and $\Omega^c$ must be accounted for in the right hand side. For that purpose, following \cite{BoNo21Constructive}, we introduce the {\it extended} stars
\begin{equation*}
\wt{S}_T^1 :=
\left\lbrace \begin{array}{rl}
  S_T^1   & \textrm{if } T \cap \partial\Omega = \emptyset,
  \\
  B_T    & \textrm{otherwise,}
\end{array}\right.
\end{equation*}
where $B_T=B(x_T,Ch_T)$ is the ball of center $x_T$ and radius $Ch_T$, with $x_T$ being the barycenter of $T$, and $C=C(\sigma)$ a shape regularity dependent constant such that $S_T^1\subset B_T$. The extended second ring $\wt{S}_T^2$ of $T$ is given by
\[
\wt{S}_T^2 := \bigcup \big\{ \wt{S}_{T'}^1 \colon T' \in \Th, T' \cap S_T^1 \neq \emptyset \big\}.
\]

The localization of the $W^s_p(\Rd)$-seminorm reads \cite[Lemma 4.1]{BoNo21Constructive}:
\begin{equation} \label{eq:localization-F-tilde}
  \|v\|_{\wt{W}^s_p(\Omega)}^p =  |v|_{W^s_p(\Rd)}^p \leq \sum_{T \in \Th} \left[ \int_T \int_{\widetilde S_T^1} \frac{|v (x) - v (y)|^p}{|x-y |^{d+sp}} \; dy \; dx +  C(d,\sigma) \frac{2^p}{s p h_T^{sp}} \| v \|^p_{L^p(T)} \right]
\end{equation}
for all $v \in \wt{W}^s_p(\Omega)$.

Our use of \eqref{eq:localization-F-tilde} will be restricted to $v$ being an interpolation error; in such a case, $v$ has vanishing means over elements and thus we can estimate the scaled $L^p$ contributions in terms of local $W^s_p$ seminorms by using Poincar\'e inequalities. 
We consider a suitable (such as Cl\'ement or Scott-Zhang) quasi-interpolation operator, 
\[
\Pi_h: \widetilde W^s_p(\Omega)\to\mathbb{V}_h ,
\]
which is stable in $W^s_p(\Omega)$ and for which one can prove the following local approximation estimates (see, for example, \cite{AcosBort2017fractional,BoNoSa18,CiarletJr}):
\begin{equation} \label{eq:interpolation} \begin{split}
& \| v - \Pi_h v \|_{L^p(T)} \le C \, h_T^{pr}  |v|_{W^r_p(S_T^1)}^p, \\
&  \int_T \int_{\widetilde S_T^1} \frac{|(v-\Pi_h v) (x) - (v-\Pi_h v) (y)|^p}{|x-y|^{d+sp}} \, d y \, d x \le C \, h_T^{p(r-s)}  |v|_{W^r_p(\wt{S}_T^2)}^p,
\end{split} \end{equation}
for all $T \in \Th$, $s \in (0,1)$, $r \in (s, 2]$, $v \in W^r_p (\wt{S}_T^2)$, where $C = C(\Omega,d,s,\sigma,right)$. 

Combining \eqref{eq:localization-F-tilde} and \eqref{eq:interpolation}, we deduce localized interpolation error estimates.

\begin{proposition}[localized interpolation estimates]\label{P:error-estimate}
Let $s \in (0,1)$, $p \in (1,\infty)$, $r \in (s, 2]$, and $\Pi_h: \widetilde W^s_p(\Omega)\to\mathbb{V}_h$ be a
quasi-interpolation operator as above.  If $v \in \widetilde W^r_p (\Omega)$, then
\begin{equation}\label{eq:interpolation-error}
\|v - \Pi_h v \|_{\widetilde W^s_p(\Omega)}^p \leq  C \left(\sum_{T \in \Th} h_T^{p(r-s)} |v|_{W^r_p(\wt{S}_T^2)}^p\right)^{\frac1p},
\end{equation}
where $C = C(\Omega,d,s,\sigma, right)$.
\end{proposition}

\subsection{Error estimates in $\widetilde{W}^s_p(\Omega)$} \label{sec:convergence}
We next derive some error estimates for the finite element solutions discussed in \S\ref{sec:FE}. We borrow techniques from the finite element analysis of classical (local) quasi-linear problems. The technique presented in \cite{GlMa75} or \cite[\S5.3]{Ciarlet} exploits the continuity and monotonicity of the operator, but not the fact that $u$ and $u_h$ solve respective minimization problems. We obtain enhanced rates by adapting an approach by Chow \cite{Chow89} for the classical $p$-Laplacian.

\begin{theorem}[error estimates]
	\label{thm:FE-error-Chow}
Let $\Omega$ be a bounded Lipschitz domain, assume that $G$ satisfies Hypothesis \ref{hyp:G}, let $p \in (1,\infty)$, $\beta \in (0,1]$ be as in such assumptions and let $p' = \frac{p}{p-1}$. Assume $f \in B^{-s+\gamma}_{p',1}(\Omega)$, where $\gamma = \max \{ \beta / p', \beta /2 \}$. Let $u$ and $u_h$ be the respective solutions of \eqref{eq:weak} and \eqref{eq:weak-discrete}. Then, if $p \in (1,2]$ it holds that 
\begin{equation} \label{eq:error_smallp}
\|u - u_h \|_{\widetilde W^s_p(\Omega)} 
\lesssim \inf_{v_h \in \mathbb{V}_h} \|u - v_h \|_{\widetilde W^s_p(\Omega)}^{p/2} 
\lesssim h^{\frac{\beta p}4 } | \log h |^{\frac12} .
\end{equation}

On the other hand, if $p \in [2,\infty)$, we have the error bound
\begin{equation} \label{eq:error_bigp}
\|u - u_h \|_{\widetilde W^s_p(\Omega)} 
\lesssim \inf_{v_h \in \mathbb{V}_h} \|u - v_h \|_{\widetilde W^s_p(\Omega)}^{2/p}
\lesssim h^{\frac{2 \beta}{p^2}} | \log h |^{\frac{2}{p^2}} .
\end{equation}
\end{theorem}
\begin{proof}
For any $v \in \widetilde{W}^s_p(\Omega)$, using either \eqref{eq:monotonicity-ple2} or \eqref{eq:monotonicity-pge2}, we write
\[ \begin{aligned}
\calF(v) - \calF(u) & = \int_0^1 \langle \calF'(u+t(v-u)) - \calF'(u), v-u \rangle \, dt \\
& = \int_0^1 \langle \calA(u+t(v-u)) - \calA u, t(v-u) \rangle \, \frac{dt}{t} \\
& \ge \left\lbrace
\begin{array}{rl}
\frac{C}{p} \|u-v\|_{\widetilde W^s_p(\Omega)}^2 \left( \|u\|_{\widetilde W^s_p(\Omega)} + \|u-v\|_{\widetilde W^s_p(\Omega)} \right)^{p-2} , & \mbox{if } p \in (1,2], \\
\frac{\alpha}{p} \|u-v\|_{\widetilde W^s_p(\Omega)}^p,  & \mbox{if } p \in [2,\infty).
\end{array}
\right. 
\end{aligned} \]

In addition, if we use either \eqref{eq:continuity-ple2} or \eqref{eq:continuity-pge2}, we obtain
\[ \begin{aligned}
\calF(v) - \calF(u) & = \int_0^1 \langle \calA(u+t(v-u)) - \calA u, t(v-u) \rangle \, \frac{dt}{t} \\
& \le \left\lbrace
\begin{array}{rl}
\frac{C}{p} \|u-v\|_{\widetilde W^s_p(\Omega)}^p , & \mbox{if } p \in (1,2], \\
\frac{C}{p} (\|u\|_{\widetilde W^s_p(\Omega)} + \|u-v\|_{\widetilde W^s_p(\Omega)})^{p-2} \|u-v\|_{\widetilde W^s_p(\Omega)}^2,  & \mbox{if }  p \in [2,\infty).
\end{array}
\right. . 
\end{aligned} \]

The proof now follows easily. Indeed, for any $v_h \in \mathbb{V}_h \subset \widetilde{W}^s_p(\Omega)$, we have, for $p \in (1, 2]$, 
\begin{equation*}
\begin{aligned}
c (\|u\|_{\widetilde W^s_p(\Omega)} + \|u-u_h\|_{\widetilde W^s_p(\Omega)} )^{p-2} \|u-u_h\|_{\widetilde W^s_p(\Omega)}^2
& \le \calF(u_h) - \calF(u) \\
& \le \calF(v_h) - \calF(u) \le c \|u-v_h\|_{\widetilde W^s_p(\Omega)}^p.
\end{aligned} 
\end{equation*}
By the stability estimates \eqref{eq:stability} and \eqref{eq:discrete_stability}, we have 
\[
\|u\|_{\widetilde W^s_p(\Omega)} + \|u-u_h\|_{\widetilde W^s_p(\Omega)} \lesssim \| f \|_{W^{-s}_{p'}(\Omega)}^{\frac1{p-1}} \lesssim  \| f \|_{B^{-s+\gamma}_{p',1}(\Omega)}^{\frac1{p-1}}
\] 
and therefore 
\[
\|u-u_h\|_{\widetilde W^s_p(\Omega)} \lesssim \|u-v_h\|_{\widetilde W^s_p(\Omega)}^{p/2}, \quad \forall v_h \in \mathbb{V}_h,
\]
which proves the first inequality in \eqref{eq:error_smallp}.

Similar considerations yield, for $p \in [2,\infty)$,
\[
\|u-u_h\|_{\widetilde W^s_p(\Omega)}^p \lesssim (\|u\|_{\widetilde W^s_p(\Omega)} + \|u-v_h\|_{\widetilde W^s_p(\Omega)} )^{p-2} \|u-v_h\|_{\widetilde W^s_p(\Omega)}^2, \quad \forall v_h \in \mathbb{V}_h,
\]
and thus the first inequality in \eqref{eq:error_bigp} holds.

We now set $v_h = \SZ u$, use the stability of $\Pi_h$ in $\widetilde W^s_p(\Omega)$, the quasi-interpolation estimate \eqref{eq:interpolation-error} and the regularity bounds \eqref{eq:sob-regularity-pLap-p>2} or \eqref{eq:sob-regularity-pLap-p<2}, to conclude:
\[
\|u-u_h\|_{\widetilde W^s_p(\Omega)} \lesssim
\begin{cases}
\|u-v_h\|_{\widetilde W^s_p(\Omega)}^{p/2} \lesssim   h^{\frac{\beta p}4 -\frac{\eps p}2} \, \eps^{-\frac12}, & \mbox{ if } p \in (1,2], \\
\|u-v_h\|_{\widetilde W^s_p(\Omega)}^{2/p} \lesssim  h^{\frac{2 \beta}{p^2} -\frac{2 \eps}p} \, \eps^{-\frac2{p^2}}, & \mbox{ if } p \in [2,\infty).
\end{cases}
\] 
for $\eps > 0$ sufficiently small. The result follows by setting $\eps = |\log h|^{-1}$.
\end{proof}

\section{Computational exploration} \label{sec:numerical}

This section presents several numerical experiments for the Dirichlet problem \eqref{eq:Dirichlet}. We use the finite element discretization discussed in Section \ref{sec:FE}
on either quasi-uniform or shape-regular graded meshes $\Th$ with grading parameter $\mu \ge 1$ satisfying
\begin{equation}
\label{eq:mesh_grad} 
h_T \approx \left\lbrace \begin{array}{ll}
C(\sigma) h^\mu, & \overline T \cap \pp \Omega \neq \emptyset \\ 
C(\sigma) h \; \dist(T,\pp\Omega)^{(\mu-1)/\mu}, & \overline T \cap \pp \Omega = \emptyset
\end{array} \right.
\end{equation}
for every $T \in \Th$. We refer to \cite{BoLiNo-Barrett} for further details on this grading strategy and additional computational experiments.

\subsection{Fractional $(p,s)$-Laplacians}
Along this section, we consider the energy minimization problem \eqref{eq:def-energy} with $G(x,y,\rho) = \frac{C_{d,s,p}}{2p} |\rho|^p$ for $p \in (1, \infty)$, that gives rise to the fractional $(p,s)$-Laplace operator \eqref{eq:def-pLap}. 

\begin{example}[boundary behavior]\label{ex:boundary-layer-psLap}
	We let $\Omega = (-0.5,0.5)^2 \setminus [0,0.5) \times (-0.5, 0]$  be an $L$-shaped domain and $f = 1$, and investigate the boundary behavior of numerical solutions. From the analytical results for the linear problem ($p = 2$), we expect the solution to have a boundary behavior of the type
	\[
	u(x) \approx \dist(x, \partial \Omega)^{\alpha(s,p)}.
	\]
We estimate the power $\alpha(s,p)$ near different points on $\partial \Omega$: 
	the mid-point of one of the edges $(0,0.5)$, a convex corner $(-0.5,0.5)$ and the reentrant corner $(0,0)$.
	We compute the numerical solutions on the graded mesh with $\mu = 2$ and $9467$ free nodes, and
 fit the power $\alpha(s,p)$ using mesh points near the boundary points mentioned above, specifically along the normal direction near the edge mid-point and along the bisectors of the angles near the corners.
 We report the results we obtain in \Cref{table:frac-Lap-2d-power-p3} and \Cref{table:frac-Lap-2d-power-p125} for $p=3$ and $p=1.75$, respectively. 
 	
	Despite the limited mesh resolution, near the edge mid-point, we observe $\alpha(s,p) \approx s$ for both $p = 3$ and $p = 1.75$. This is consistent with the behavior shown in \cite{Iannizzotto:14,MR4109087} for domains satisfying an exterior ball condition.  In addition, we  notice that $\alpha(s,p) > s$ near the convex corner, $\alpha(s,p) < s$ near the reentrant corner, and the deviation from $s$ is larger when $p = 1.75$ compared to $p = 3$.

	\begin{table}[htbp]
	\centering\scalebox{0.95}{
		\begin{tabular}{|c |c| c| c| c| c| c| c| c| c| c|}
			\hline 
			Value of $s$ & 0.1 & 0.2 & 0.3 & 0.4 & 0.5 & 0.6 & 0.7 & 0.8 & 0.9 \\ \hline
			Edge mid-point	 & 0.11 & 0.21 & 0.31 & 0.41 & 0.51 & 0.61 & 0.70 & 0.80 & 0.90  \\ \hline  		
			Convex corner	& 0.09 & 0.22 & 0.35 & 0.49 & 0.64 & 0.79 & 0.94 & 1.11 & 1.29  \\ \hline
			Reentrant corner & 0.06 & 0.16 & 0.24 & 0.31 & 0.39 & 0.46 & 0.54 & 0.62 & 0.70  \\ \hline   		  				
		\end{tabular}
	}
        \vskip0.2cm
	\caption{Example \eqref{ex:boundary-layer-psLap}: Exponents $\alpha=\alpha(s,p)$ of boundary asymptotics
        $u(x) \approx \dist(x,\partial\Omega)^\alpha$ for $p = 3$, different values of fractional order $s$, and three qualitatively distinct boundary points.}
	\label{table:frac-Lap-2d-power-p3}
\end{table}

\begin{table}[htbp]
	\centering\scalebox{0.95}{
		\begin{tabular}{|c |c| c| c| c| c| c| c| c| c| c|}
			\hline 
			Value of $s$ & 0.1 & 0.2 & 0.3 & 0.4 & 0.5 & 0.6 & 0.7 & 0.8 & 0.9 \\ \hline
			Edge mid-point	 & 0.10 & 0.21 & 0.31 & 0.41 & 0.51 & 0.60 & 0.70 & 0.79 & 0.89  \\ \hline	
			Convex corner	&  0.11 & 0.27 & 0.46 & 0.66 & 0.86 & 1.07 & 1.29 & 1.51 & 1.76   \\ \hline
			Reentrant corner & 0.07 & 0.12 & 0.17 & 0.23 & 0.29 & 0.35 & 0.42 & 0.49 & 0.57  \\ \hline	
		\end{tabular}
	}
        \vskip0.2cm
	\caption{Example \eqref{ex:boundary-layer-psLap}: Exponents $\alpha=\alpha(s,p)$ of boundary asymptotics $u(x) \approx \dist(x,\partial\Omega)^\alpha$ for $p = 1.75$, different values of fractional order $s$, and three qualitatively distinct boundary points.}
	\label{table:frac-Lap-2d-power-p125}
\end{table}	
\end{example}

\begin{example}[convergence rates]\label{ex:rates-psLap}
Consider the square domain $\Omega = (-0.5,0.5)^2$ and $f = 1$. Since $f$ is smooth, \Cref{C:Sobolev-delta} (maximal Sobolev regularity) gives  $u \in \wt{W}^{s + 1/p - \eps}_p(\Omega)$ for $p \ge 2$ and $u \in \wt{W}^{s + 1/2 - \eps}_p(\Omega)$ for $1 < p \le 2$. We compute numerical solutions for $p = 3$, $p=1.75$ and different values of $s$ on quasi-uniform meshes, and examine convergence rates in the energy norm. Since the exact solutions $u$ are unknown, we use $\Vert u_h - u_{h/2} \Vert_{\wt{W}^s_p(\Omega)}$ as a proxy for $\Vert u_h - u \Vert_{\wt{W}^s_p(\Omega)}$. 
\Cref{table:frac-Lap-2d-sq-rates} summarizes our findings.

\begin{table}[htbp]
	\centering\scalebox{0.95}{
		\begin{tabular}{|c |c| c| c| c| c| c| c| c| c| c|}
			\hline 
			Value of $s$ & 0.1 & 0.2 & 0.3 & 0.4 & 0.5 & 0.6 & 0.7 & 0.8 & 0.9 \\ \hline
			$p = 3$ &  0.326 & 0.335 & 0.333 & 0.329 & 0.328 & 0.329 & 0.335 & 0.357 & 0.494 \\ \hline	
			$p = 1.75$ &  0.558 & 0.552  & 0.555 & 0.561 & 0.569 & 0.583 & 0.607 & 0.658 & 0.790 \\ \hline
		\end{tabular}
	}
        \vskip0.2cm
	\caption{\Cref{ex:rates-psLap}: Convergence rates on uniform meshes for $p = 1.75, 3$ and different values of $s$. They indicate that the theoretical rates in \Cref{thm:FE-error-Chow} (error estimates) might be suboptimal.}
	\label{table:frac-Lap-2d-sq-rates}
\end{table}	

The rates are approximately $1/p \approx 0.33$ for $p = 3$ except for the case $s = 0.9$, where we believe the discrepancy is due to the proxy solution not being sufficiently refined in comparison to the rest of the experiments. Although the rate $1/p$ is larger than the theoretical rate $2/p^2$ of \eqref{eq:error_bigp} in \Cref{thm:FE-error-Chow} (error estimates), we point out that it is consistent with the best approximation error
\[
\inf_{v_h \in \mathbb{V}_h} \|u - v_h \|_{\widetilde W^s_p(\Omega)} \lesssim
h^{1/p} | \log h |^{1/p}.
\]
This indicates that, instead of the regularity of $u$,
 the suboptimal rates in the theory might be a consequence of the suboptimal power $2/p$ in the error estimate of \eqref{eq:error_bigp}
\[
\|u - u_h \|_{\widetilde W^s_p(\Omega)} 
\lesssim \inf_{v_h \in \mathbb{V}_h} \|u - v_h \|_{\widetilde W^s_p(\Omega)}^{2/p}.
\]

Similarly, for $p = 1.75$, we observe that the rates are approximately $1/p \approx 0.57$ except for $s = 0.7, 0.8, 0.9$, where we believe the meshes are not fine enough to deliver accurate rates. This indicates that instead of the regularity $u \in \wt{W}^{s + 1/2 - \eps}_p(\Omega)$ proved in \Cref{C:Sobolev-delta}, the solution $u$ in this example might satisfy $u \in \wt{W}^{s + 1/p - \eps}_p(\Omega)$. In addition, the power $p/2$ in the error estimate \eqref{eq:error_smallp} of \Cref{thm:FE-error-Chow},
\[
\|u - u_h \|_{\widetilde W^s_p(\Omega)} 
\lesssim \inf_{v_h \in \mathbb{V}_h} \|u - v_h \|_{\widetilde W^s_p(\Omega)}^{p/2},
\] 
might not be optimal as well. These two reasons together lead to the theoretical suboptimal rate $p/4$
of \eqref{eq:error_smallp}.
\end{example}

\subsection{Linear operators with finite horizon} \label{sec:exp-var}

We consider the operator $\calA_\delta u$ defined in \eqref{eq:FH-operator} and let $p = 2$, that gives rise to the linear fractional Laplacian with finite horizon and variable diffusivity.

\begin{example}[truncated fractional Laplacian in $1$D]\label{ex:boundary-1d-var}
	Consider $\Omega = (-1, 1)$, $p = 2, \delta = 0.2, f = \frac{\delta^{2-2s}}{1-s}$ and $\psi = \chi_{[0,1]}$ where $\chi_I$ is the characteristic function of $I$. The resulting operator is the truncated fractional Laplacian we discussed in  \Cref{R:Sobolev-delta}.
	We first compute numerical solutions for different $s$ using a mesh graded according to \eqref{eq:mesh_grad} with $h = 2^{-12},$ $\mu = 2$ to investigate the boundary behavior of solutions.  \Cref{fig:frac-Lap-var-1d} displays the solutions we obtained for several values of $s$.
	\begin{figure}[!htb]
		\begin{center}
			\includegraphics[width=0.7\linewidth]{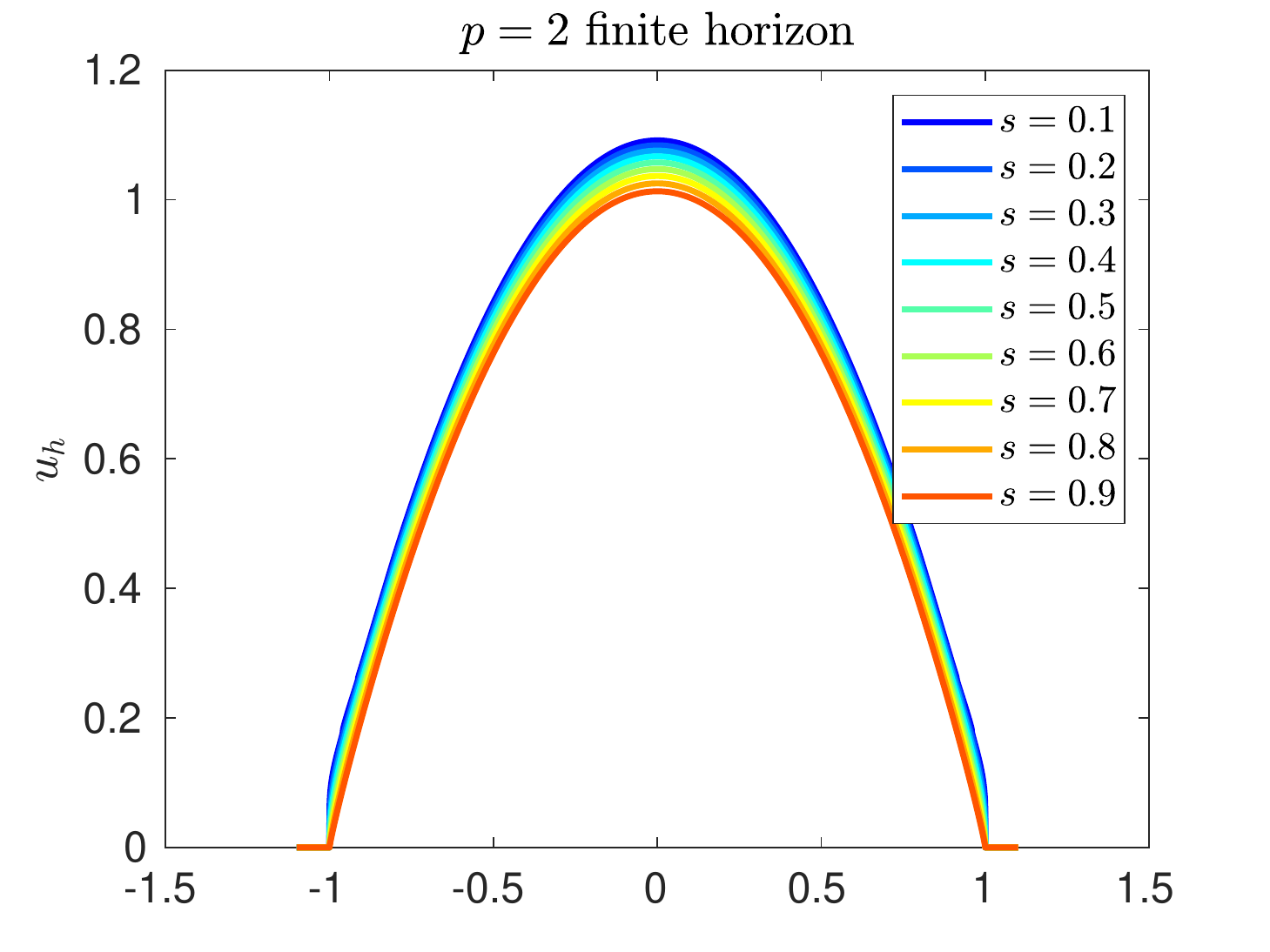}
		\end{center}
		\caption{Example \ref{ex:boundary-1d-var}: Numerical solutions of linear operator with finite horizon for different values of $s$.}
		\label{fig:frac-Lap-var-1d}
	\end{figure}
	
Assuming the solutions have an algebraic boundary behavior
	\[
	u(x) \approx \dist(x, \partial \Omega)^{\alpha(s)},
	\]
we estimate $\alpha(s)$ numerically in \Cref{table:frac-Lap-var-1d-bd-exponent}. We clearly observe $\alpha(s) \approx s$; this is the same as for the fractional Laplacian $(-\Delta)^s$.
	\begin{table}[!htbp]
		\centering\scalebox{0.95}{
			\begin{tabular}{|c |c| c| c| c| c| c| c| c| c| c|}
				\hline 
				Value of $s$ & 0.1 & 0.2 & 0.3 & 0.4 & 0.5 & 0.6 & 0.7 & 0.8 & 0.9 \\ \hline
				$\alpha(s)$ & 0.10 & 0.19 & 0.29 & 0.39 & 0.49 & 0.60 & 0.70 & 0.80 & 0.90    \\ \hline		  				
			\end{tabular}
		}
                \vskip0.2cm
		\caption{Example \ref{ex:boundary-1d-var}: Boundary exponents $\alpha(s)$ for different values of $s$. They confirm $u(x) \approx \dist (x,\partial\Omega)^s$.}
		\label{table:frac-Lap-var-1d-bd-exponent}
	\end{table}	

	Next, we measure convergence rates for different $s$ on uniform meshes for $h$ from $2^{-8}$ to $2^{-12}$. Since we do not know a closed formula for the solution $u$, we use $\Vert u_h - u_{h/2} \Vert_{\wt{H}^s(\Omega)}$ as a proxy for $\Vert u_h - u \Vert_{\wt{H}^s(\Omega)}$ and present the rates in terms of $h$ in \Cref{table:frac-Lap-var-1d-rates}. We observe the convergence rates are about $0.5$ for all $s$, in agreement with the regularity $u \in \wt{H}^{s+\frac12-\eps}(\Omega)$ proved in \Cref{C:FH} (maximal Sobolev regularity) and a standard best-approximation argument. 
	\begin{table}[htbp]
	\centering\scalebox{0.95}{
		\begin{tabular}{|c |c| c| c| c| c| c| c| c| c| c|}
			\hline 
			Value of $s$ & 0.1 & 0.2 & 0.3 & 0.4 & 0.5 & 0.6 & 0.7 & 0.8 & 0.9 \\ \hline
			Rate &  0.530 & 0.508 & 0.500 & 0.498 & 0.497 & 0.498 & 0.498 & 0.498 & 0.501 \\ \hline	
		\end{tabular}
	}
        \vskip0.2cm
	\caption{\Cref{ex:boundary-1d-var}: Optimal convergence rates on uniform meshes for different values of $s$.}
	\label{table:frac-Lap-var-1d-rates}
\end{table}	
\end{example}

\bibliographystyle{abbrv}
\bibliography{ps-lap.bib}
\end{document}